\documentclass[10pt]{amsart}
\usepackage{amssymb}
\usepackage{amscd}
\usepackage[all]{xy}
\usepackage{tikz}
\usetikzlibrary{patterns}
\usepackage{comment}
\usepackage{mathabx}

\numberwithin{equation}{section}

\def\today{\number\day\space\ifcase\month\or   January\or February\or
   March\or April\or May\or June\or   July\or August\or September\or
   October\or November\or December\fi\   \number\year}

\newcounter{TmpEnumi}

\theoremstyle{definition}
\newtheorem{thm}{Theorem}[section]
\newtheorem{lem}[thm]{Lemma}
\newtheorem{prp}[thm]{Proposition}
\newtheorem{dfn}[thm]{Definition}
\newtheorem{cor}[thm]{Corollary}

\newtheorem{rmk}[thm]{Remark}

\newtheorem{pbm}[thm]{Problem}

\newtheorem{cns}[thm]{Construction}

\newcommand{\beq}{\begin{equation}}
\newcommand{\eeq}{\end{equation}}
\newcommand{\beqr}{\begin{eqnarray*}}
\newcommand{\eeqr}{\end{eqnarray*}}
\newcommand{\bal}{\begin{align*}}
\newcommand{\eal}{\end{align*}}
\newcommand{\bei}{\begin{itemize}}
\newcommand{\eei}{\end{itemize}}
\newcommand{\limi}[1]{\lim_{{#1} \to \infty}}

\newcommand{\af}{\alpha}
\newcommand{\bt}{\beta}
\newcommand{\gm}{\gamma}

\newcommand{\ep}{\varepsilon}

\newcommand{\et}{\eta}

\newcommand{\io}{\iota}
\newcommand{\te}{\theta}
\newcommand{\ld}{\lambda}
\newcommand{\sm}{\sigma}
\newcommand{\kp}{\kappa}
\newcommand{\ph}{\varphi}
\newcommand{\ps}{\psi}
\newcommand{\rh}{\rho}
\newcommand{\om}{\omega}
\newcommand{\ta}{\tau}

\newcommand{\Dt}{\Delta}

\newcommand{\Sm}{\Sigma}

\newcommand{\Om}{\Omega}

\newcommand{\Q}{{\mathbb{Q}}}
\newcommand{\Z}{{\mathbb{Z}}}
\newcommand{\R}{{\mathbb{R}}}
\newcommand{\C}{{\mathbb{C}}}
\newcommand{\N}{{\mathbb{Z}}_{> 0}}
\newcommand{\Nz}{{\mathbb{Z}}_{\geq 0}}

\newcommand{\chH}{\widecheck{H}}

\newcommand{\id}{{\operatorname{id}}}
\newcommand{\ev}{{\operatorname{ev}}}
\newcommand{\sint}{{\operatorname{int}}}

\newcommand{\diag}{{\operatorname{diag}}}
\newcommand{\supp}{{\operatorname{supp}}}
\newcommand{\rank}{{\operatorname{rank}}}

\newcommand{\Aut}{{\operatorname{Aut}}}

\newcommand{\Aff}{{\operatorname{Aff}}}
\newcommand{\rc}{{\operatorname{rc}}}

\newcommand{\T}{{\operatorname{T}}}
\newcommand{\Ker}{{\operatorname{Ker}}}
\newcommand{\drr}{{\operatorname{drr}}}

\newcommand{\cZ}{{\mathcal{Z}}}

\newcommand{\dirlim}{\varinjlim}

\newcommand{\Mi}{M_{\infty}}

\newcommand{\andeqn}{\qquad {\mbox{and}} \qquad}

\newcommand{\ifo}{if and only if}

\newcommand{\ca}{C*-algebra}
\newcommand{\uca}{unital C*-algebra}

\newcommand{\hm}{homomorphism}
\newcommand{\im}{isomorphism}

\newcommand{\am}{automorphism}
\newcommand{\fd}{finite dimensional}
\newcommand{\tst}{tracial state}
\newcommand{\hsa}{hereditary subalgebra}

\newcommand{\pj}{projection}

\newcommand{\nzp}{nonzero projection}
\newcommand{\mvnt}{Murray-von Neumann equivalent}

\newcommand{\ct}{continuous}
\newcommand{\cfn}{continuous function}
\newcommand{\nbhd}{neighborhood}
\newcommand{\cms}{compact metric space}

\newcommand{\chs}{compact Hausdorff space}
\newcommand{\hme}{homeomorphism}
\newcommand{\mh}{minimal homeomorphism}

\newcommand{\cp}{crossed product}

\renewcommand{\S}{\subset}
\newcommand{\ov}{\overline}

\newcommand{\I}{\infty}
\newcommand{\E}{\varnothing}

\newcommand{\lrc}{\operatorname{lrc}}
\newcommand{\pt}{\mathrm{pt}}
\newcommand{\Ell}{\operatorname{Ell}}

\title[Same Elliott invariant and radius of comparison]{Simple
 AH~algebras with the same Elliott invariant and radius of
 comparison}

\author{Ilan Hirshberg and N.~Christopher Phillips}

\date{27~June 2024}

\address{Department of Mathematics, Ben-Gurion University of the Negev,
  Be'er Sheva, Israel.}

\address{Department of Mathematics, University  of Oregon,
       Eugene OR 97403-1222, USA.}

\thanks{This material is based upon work partially supported
 by the US National Science Foundation under Grant DMS-2055771,
 the US-Israel Binational Science Foundation,
 and The Fields Institute for Research in Mathematical Sciences.}

\begin{document}

\begin{abstract}
We construct an uncountable family of pairwise nonisomorphic
AH algebras with the same Elliott invariant and same radius of comparison.
They can be distinguished by a local radius of comparison function,
naturally defined on the positive cone of the $K_0$ group.
\end{abstract}

\maketitle

\indent
Classification theory for simple nuclear C*-algebras reached an important milestone in the last decade.
The results of \cite{EGLN,TWW}, capping decades of work by many authors,
show that simple separable unital C*-algebras  which are $\cZ$-stable and satisfy the Universal Coefficient Theorem are classified via the Elliott invariant $\Ell (\cdot )$,
which consists of the ordered $K_0$-group along with the class of the identity,
the $K_1$-group, the trace simplex,
and the pairing between the trace simplex and the $K_0$-group.
Earlier examples due to Toms (\cite{Tms2}) and R{\o}rdam (\cite{Ror02}),
using ideas introduced by Villadsen (\cite{Villadsen}),
show that one cannot extend this classification theorem beyond the $\cZ$-stable case without either extending the invariant or restricting to another class of C*-algebras.
Intriguingly, it was shown recently (\cite{ELN}) that
AH algebras constructed using the recipe introduced in \cite{Villadsen}
(so-called Villadsen algebras),
which furthermore use finite dimensional contractible seed spaces
satisfying additional technical conditions,
are in fact classified by the combination of the Elliott invariant
and the radius of comparison.
The radius of comparison
is a numerical invariant introduced by Toms in \cite{Tms1},
which measures how badly comparison of positive elements fails
(or, in a sense, how far a C*-algebra is far from being $\cZ$-stable).
The goal of this paper is to show that the combination
of the Elliott invariant and the radius of comparison cannot
classify simple AH algebras:
we provide an uncountable family of pairwise nonisomorphic AH algebras
with the same Elliott invariant and the same radius of comparison.
In particular, adding to the Elliott invariant one number
naturally associated to the algebra seems insufficient to classify simple
unital AH~algebras.

One of the important facets of classification theory is existence theorems for homomorphisms,
namely, any map between the Elliott invariants is induced by a homomorphism between the C*-algebras,
and isomorphisms between the Elliott invariants are induced by isomorphisms of the C*-algebras.
In~\cite{HP_asymmetry},
we provided an example of an AH algebra $A$ along with an automorphism of $\Ell(A)$ which is not induced by an automorphism of $A$.
This paper is a continuation of \cite{HP_asymmetry},
and uses similar ideas but with more refined control over the outcome.
In particular, the examples we construct here
share that feature of the examples constructed in \cite{HP_asymmetry}:
each one also has an automorphism of the Elliott invariant
which is not induced by an automorphism of the C*-algebra.
(See Corollary~\ref{C_24626_New}.)
The invariant we use to distinguish them is a \emph{local radius of comparison} function $\lrc \colon K_0(A)_{+} \to [0,\infty]$,
which is well defined for C*-algebras with cancellation, given by
\[
\lrc ( [p] ) = \rc(p(K \otimes A)p) \, .
\]
 (For C*-algebras without cancellation,
the local radius of comparison function is well defined as a function
on the Murray-von Neumann semigroup of projections.)
See Definition~\ref{D_3Y14_lrc} and Corollary~\ref{C_3Y14_lrc_K0}.

We now give an overview of our construction.
 We start with the counterexample
 from \cite[Theorem 1.1]{Tms2}.
 We consider two direct systems, described diagrammatically as follows:
\[
\xymatrix{
  C (X_0) \ar@<-1ex>@{.>}[r] \ar@<-0.7ex>@{.>}[r]  \ar[r] \ar@<0.5ex>[r] \ar@<1ex>[r] &
  C (X_1)\otimes M_{r (1)} \ar@<-1ex>@{.>}[r] \ar@<-0.7ex>@{.>}[r]
  \ar[r] \ar@<.5ex>[r] \ar@<1ex>[r] &
  C (X_2)\otimes M_{r (2)} \ar@<-1ex>@{.>}[r] \ar@<-0.7ex>@{.>}[r]
  \ar[r] \ar@<.5ex>[r] \ar@<1ex>[r] & \cdots \\
  C (Y_0) \ar@<.9ex>@{.>}[r]
  \ar[r] \ar@<-.5ex>[r] \ar@<-1ex>[r] &
  C (Y_1)\otimes M_{r (1)}\ar@<0.9ex>@{.>}[r]
  \ar[r] \ar@<-0.5ex>[r] \ar@<-1ex>[r] &
  C (Y_2)\otimes M_{r (2)} \ar@<.9ex>@{.>}[r]
  \ar[r] \ar@<-.5ex>[r] \ar@<-1ex>[r] & \cdots }
\]
  The spaces in the diagram are contractible CW complexes
whose dimensions increase rapidly
 compared to the sizes of the matrix algebras.
(In fact, we will take $Y_0 = X_0$, but $Y_n \neq X_n$ for $n \geq 1$.)
 The direct system is constructed so as to have
 positive radius of comparison.
 The ordinary arrows indicate a large (and rapidly increasing)
 number of embeddings which are carefully chosen,
 and the dotted arrows indicate a small number of point evaluation maps,
 put in so as to ensure that the resulting direct limit is simple.
The two direct systems
 are not identical: one has more point evaluations than the other,
so as to produce a direct
 limit with a lower radius of comparison.
In the diagram, this is indicated by
more dotted arrows in the top direct system than the bottom one.

 Our construction involves moving some of the point evaluations across,
 so as to merge the two systems,
 getting:
\[
\xymatrix{
 C (X_0) \ar@{.>}[dr] \ar@<-0.7ex>@{.>}[r]  \ar[r] \ar@<.5ex>[r] \ar@<1ex>[r]
 & C (X_1)\otimes M_{r (1)}\ar@{.>}[dr] \ar@<-0.7ex>@{.>}[r]
 \ar[r] \ar@<.5ex>[r] \ar@<1ex>[r]
 & C (X_2)\otimes M_{r (2)} \ar@{.>}[dr] \ar@<-0.7ex>@{.>}[r]
 \ar[r] \ar@<.5ex>[r] \ar@<1ex>[r] & \cdots
 \\
 C (Y_0) \ar@{.>}[ur]  \ar[r] \ar@<-.5ex>[r]
 \ar@<-1ex>[r] & C (Y_1) \otimes M_{r (1)} \ar@{.>}[ur]
 \ar[r] \ar@<-.5ex>[r]
 \ar@<-1ex>[r]& C (Y_2) \otimes M_{r (2)} \ar@{.>}[ur]
 \ar[r] \ar@<-.5ex>[r] \ar@<-1ex>[r] & \cdots .}
\]
The result is a simple AH algebra.
However, as most of the arrows are the same as those in the first diagram,
which produces a direct sum, we can show that some features of the
nonsimple direct limit of the first diagram carry over to the simple direct limit of the second diagram.
Specifically, we can show that cutdowns by the projection
which is $1$ on the top copy of $C(X_0)$ and $0$ on the bottom copy,
and its orthogonal complement,
produce corners with different radii of comparison.
By making careful choices so as to control the $K$-theory,
we can make those two corners distinguished,
in a way which does not depend on the choice of the direct system.
By choosing other parameters appropriately,
we use this idea to produce an uncountable family of examples with the same radius of comparison,
but which can be distinguished from one another by looking at the radii of comparison of those corners.

We note that the idea of merging diagrams,
first used in \cite{HP_asymmetry},
was used later to construct interesting examples of automorphisms of C*-algebras with positive radius of comparison,
which is hard to do directly for Villadsen's construction (\cite{AGP,Hrs_exotic}),
and its ``dual'' was used to similar effect in~\cite{AsVa}.

 The paper is organized as follows.
In Section~\ref{Sec_3X28_GVSys},
we give a general construction of what we call a \emph{multiseed Villadsen system},
which involves several merged diagrams as above,
and compute the radius of comparison of the resulting direct limit.
In Section~\ref{Sec_K_theory},
we restrict ourselves for simplicity to the case of two seeds which are contractible,
and give an explicit description of the ordered $K_0$-group of the direct limit.
We then describe the trace space of those direct limits
and its pairing with $K_0$,
assuming some technical conditions on the connecting maps and the seed spaces.
In Section~\ref{S_lrc} we define the local radius of comparison
and prove Theorem~\ref{T_3Y14_main},
which is the main result.

\section{A generalization of Villadsen systems}\label{Sec_3X28_GVSys}

We provide here the general set-up for constructing C*-algebras
obtained by merging finitely many direct systems arising
from the construction of Villadsen algebras,
along with a computation of the radius of comparison
of the resulting direct limit.

\indent
We denote by $\T (A)$ the tracial state space of a \uca~$A$.
Any $\tau \in \T (A)$ extends to a non-normalized trace
on $M_n \otimes A$ for all $n \in \N$, given by
$\mathrm{Tr} \otimes \tau$.
We simplify notation by using $\tau$ for this extended trace as well.
For $a \in M_n (A)_{+}$,
we define $d_{\tau}(a) = \lim_{n \to \infty} \tau (a^{1/n})$.
Given $r \in [0, \I]$,
a unital C*-algebra $A$ is said to have $r$-comparison
if for any $n \in \N$ and $a, b \in M_n (A)_{+}$
satisfying $d_{\tau} (a) + r  < d_{\tau}(b)$ for all $\tau \in \T (A)$,
we have $a \precsim b$
(that is, that $a$ is Cuntz subequivalent to $b$).
The radius of comparison, denoted $\rc(A)$,
is the infimum of $r \in [0, \I]$ such that $A$ has $r$-comparison.
(In general, one needs to consider quasitraces and not just traces.
However, for nuclear C*-algebras all quasitraces are traces,
by a theorem of Haagerup, \cite[Theorem 5.11]{haagerup}.)
Also, for algebras that have quotients that are not stably finite,
the number $r_A$ described after \cite[Definition 3.2.2]{BRTTW}
is more suitable.
However, no such algebras will appear in this paper.

We recall several standard definitions for convenience.

\begin{dfn}\label{D_3X28_Diag_map}
Let $m \in \N$, let $X_1, X_2, \ldots, X_m, Y$ be \chs{s},
and let $k_1, k_2, \ldots k_m, l \in \N$.
A unital \hm{}
\[
\ph \colon \bigoplus_{j = 1}^{m} C (X_{j}, M_{k_j}) \to C (Y, M_l)
\]
is {\emph{diagonal}} if for $j = 1, 2, \ldots, m$ there are $t_j \in \Nz$
and \ct{} maps
\[
h_{j, 1}, h_{j, 2}, \ldots, h_{j, t_j} \colon Y \to X_j
\]
such that
$\ph$ is unitarily equivalent to the \hm{} $\ps$ which sends
\[
(f_1, f_2, \ldots, f_m) \in \bigoplus_{j = 1}^{m} C (X_{j}, M_{k_j})
\]
to
\[
\begin{split}
\ps (f_1, f_2, \ldots, f_m)
& = \diag \bigl( f_1 \circ h_{1, 1}, \, f_1 \circ h_{1, 2}, \, \ldots,
    \, f_1 \circ h_{1, t_1},
\\
& \hspace*{3em} {\mbox{}}
    \, f_2 \circ h_{2, 1}, \, f_2 \circ h_{2, 2}, \, \ldots,
    \, f_2 \circ h_{2, t_2}, \, \ldots,
\\
& \hspace*{3em} {\mbox{}}
    \, f_m \circ h_{m, 1}, \, f_m \circ h_{m, 2}, \, \ldots,
    \, f_m \circ h_{m, t_m} \bigr).
\end{split}
\]
\end{dfn}

\begin{dfn}\label{D_3X28_Diag_sys}
Let $\bigl( (A_n)_{n \in \Nz}, (\ph_{n_2, n_1})_{n_2 \geq n_1} \bigr)$
be an AH direct system, in which for each $n \in \Nz$ there
are $m (n) \in \N$, \chs{s} $X_{n, 1}, X_{n, 2}, \ldots, X_{n, m (n)}$,
and $l_{n, 1}, l_{n, 2}, \ldots, l_{n, m (n)} \in \N$
such that $A_n = \bigoplus_{j = 1}^{m (n)} C (X_{n, j}, M_{l_{n, j}})$.
We say that $\bigl( (A_n)_{n \in \Nz}, (\ph_{n_2, n_1})_{n_2 \geq n_1} \bigr)$
{\emph{has diagonal maps}} if for every $n \in \Nz$ the \hm{}
$\ph_{n + 1, \, n} \colon A_n \to A_{n + 1}$
is diagonal in the sense of Definition \ref{D_3X28_Diag_map}.
\end{dfn}

\begin{lem}\label{L_3X31_ranks}
Let $m \in \N$, let $X_1, X_2, \ldots, X_m, Y$ be \chs{s},
and let $k_1, k_2, \ldots k_m, l \in \N$.
Let
\[
\ph \colon \bigoplus_{j = 1}^{m} C (X_{j}, M_{k_j}) \to C (Y, M_l)
\]
be a diagonal unital \hm.
Let $\af, \bt \in [0, \I]$, and for $j = 1, 2, \ldots, m$
let $p_j \in \Mi ( C (X_{j}, M_{k_j}) )$ be a constant \pj{}
whose rank satisfies $\af k_j < \rank (p_j) < \bt k_j$.
Set $p = (p_1, p_2, \ldots, p_m)$.
Then $\ph (p) \in \Mi ( C (Y, M_l) )$
is a constant \pj{} whose rank satisfies $\af l < \rank (\ph (p)) < \bt l$.
\end{lem}

\begin{proof}
Clearly diagonal \hm{s} send constant functions to constant functions.

For $j = 1, 2, \ldots, m$ let $r_j$ be the multiplicity of the
$j$-th partial map
\[
\ph_j \colon C (X_{j}, M_{k_j}) \to C (Y, M_l).
\]
Then $\af k_j r_j < \rank (\ph_j (p_j)) < \bt k_j r_j$.
Sum over $j = 1, 2, \ldots, m$ and use $\sum_{j = 1}^{m} k_j r_j = l$.
\end{proof}

Without diagonal maps, the rank computation is the same.
However, there is no reason for $\ph (p)$ to be a constant \pj.

The next definition is based on Definition~2.1 of~\cite{Tms1}.

\begin{dfn}\label{D_3X28_drr}
Let $\bigl( (A_n)_{n \in \Nz}, (\ph_{n_2, n_1})_{n_2 \geq n_1} \bigr)$
be an AH direct system, in which for each $n \in \Nz$ there
are $m (n) \in \N$, \chs{s} $X_{n, 1}, X_{n, 2}, \ldots, X_{n, m (n)}$,
$l_{n, 1}, l_{n, 2}, \ldots, l_{n, m (n)} \in \N$,
and \pj{s} $p_{n, j} \in C (X_{n, j}, K)$
for $j = 1, 2, \ldots, m (n)$, each with constant rank,
such that
$A_n = \bigoplus_{j = 1}^{m (n)} p_{n, j} C (X_{n, j}, K) p_{n, j}$.
The {\emph{dimension rank ratio}} of this system
is defined to be
\[
\drr \bigl( (A_n)_{n \in \Nz}, (\ph_{n_2, n_1})_{n_2 \geq n_1} \bigr)
 = \limsup_{n \to \I} \max_{1 \leq j \leq m (n)}
       \frac{\dim (X_{n, j})}{\rank (p_{n, j})}.
\]
\end{dfn}

In Definition~2.1 of~\cite{Tms1},
$\drr (A)$ is defined to be the infimum,
over all AH~systems as in Definition~\ref{D_3X28_drr}
whose direct limit is~$A$,
of what we have called the dimension rank ratio of the system.
Since a system can always be replaced by a subsystem
for which the limit in Definition~\ref{D_3X28_drr} is,
in terms of the original system,
\[
\liminf_{n \to \I} \max_{1 \leq j \leq m (n)}
  \frac{\dim (X_{n, j})}{\rank (p_{n, j})},
\]
one can use $\liminf$ in the definition of $\drr (A)$.

\begin{cns}\label{D_3X28_GVS_df}
We give a recipe for a special type of AH~direct system
$\bigl( (A_n)_{n \in \Nz}, (\ph_{n_2, n_1})_{n_2 \geq n_1} \bigr)$.
The main special feature is Condition~(\ref{I_3X30_GVS_Vlds}).
\begin{enumerate}
\item\label{I_3X28_GVS_m}
Let $m \in \N$ be a positive integer.
(This is the number of ``seed spaces''.)
\item\label{I_3X28_GVS_Spaces}
Let $X_1, X_2, \ldots, X_m$ be compact metric spaces.
(These are the ``seed spaces''.)
\item\label{I_3Y23_GVS_Amp}
For $j = 1, 2, \ldots, m$, let $(d_j (n))_{n \in \N}$
be a sequence in~$\N$,
and for $n \in \Nz$ define $s_j (n) = \prod_{k = 1}^{n} d_j (k)$.
By a standard convention, $s_j (0) = 1$.
(The spaces at level~$n$ will be
$X_{1}^{s_1 (n)}, \, X_{2}^{s_2 (n)}, \, \ldots, \, X_{m}^{s_m (n)}$.)
\item\label{I_3X28_GVS_Sizes}
For $n \in \N$,
let $(\mu_{k, j} (n))_{j, k = 1}^{m}$ be an $m \times m$
matrix of nonnegative integers such that, for $j = 1, 2, \ldots, m$,
we have $1 \leq d_j (n) \leq \mu_{j, j} (n)$.
(These are the partial embedding multiplicities.)
\item\label{I_3X30_GVS_Cumulative}
Let $r_1 (0), r_2 (0), \ldots, r_m (0) \in \N$.
(These are the matrix sizes at the initial level.)
For $j = 1, 2, \ldots, m$ and $n \in \Nz$,
we recursively define
\[
r_{k} (n + 1) = \sum_{j = 1}^m \mu_{k, j} (n + 1) r_j (n).
\]
(These are the matrix sizes at level $n + 1$.)
\item\label{I_3X30_GVS_Growth}
We require that for $j = 1, 2, \ldots, m$, we have $\limi{n} r_j (n) = \I$.
\item\label{I_3X28_GVS_algs}
For $n \in \Nz$, define
\[
A_n = \bigoplus_{j = 1}^{m} C \bigl( X_{j}^{s_j (n)}, \, M_{r_j (n)} \bigr).
\]
\item\label{I_3X28_GVS_maps}
For $n \in \Nz$, let
$\ph_{n + 1, \, n} \colon A_n \to A_{n + 1}$ be a unital \hm,
satisfying the conditions in (\ref{I_3Y25_GVS_Mult})
and~(\ref{I_3X30_GVS_Vlds}) below.
\item\label{I_3X30_GVS_Moremaps}
For $n_1, n_2 \in \Nz$ with $n_2 \geq n_1$, we define
\[
\ph_{n_2, n_1}
 = \ph_{n_2, \, n_2 - 1} \circ \ph_{n_2 - 1, \,  n_2 - 2}
    \circ \cdots \circ \ph_{n_1 + 1, \, n_1}.
\]
Further, for $n \in \Nz$ let $\ph_{\I, n} \colon A_n \to \dirlim_k A_k$
be the canonical map associated with the direct limit.
\item\label{I_3X30_GVS_Part}
For $n_1, n_2 \in \Nz$ with $n_2 \geq n_1$
and $j_1, j_2 \in \{ 1, 2, \ldots, m \}$, let
\[
\ph_{n_2, n_1}^{j_2, j_1} \colon
  C \bigl( X_{j_1}^{s_{j_1} (n_1)}, \, M_{r_{j_1} (n_1)} \bigr)
   \to C \bigl( X_{j_2}^{s_{j_2} (n_2)}, \, M_{r_{j_2} (n_2)} \bigr)
\]
be the $(j_1, j_2)$ partial map, that is, the composition
\[
C \bigl( X_{j_1}^{s_{j_1} (n_1)}, \, M_{r_{j_1} (n_1)} \bigr)
  \longrightarrow A_{n_1}
  \stackrel{\ph_{n_2, n_1}}{\longrightarrow} A_{n_2}
  \longrightarrow
     C \bigl( X_{j_2}^{s_{j_2} (n_2)}, \, M_{r_{j_2} (n_2)} \bigr),
\]
is which the first and last maps are the standard inclusion
and projection maps.
\item\label{I_3Y25_GVS_Mult}
We require that the partial multiplicities of the maps
$\ph_{n + 1, \, n} \colon A_n \to A_{n + 1}$
be given by the matrix $(\mu_{k, j} (n + 1))_{j, k = 1}^{m}$.
That is, for $n \in \Nz$ and $j, k \in \{ 1, 2, \ldots, m \}$,
we have
$\rank \bigl( \ph_{n + 1, n}^{k, j} (1) \bigr) = \mu_{k, j} (n + 1) r_j (n)$.
\item\label{I_3X30_GVS_Vlds}
For $n \in \Nz$ and $j = 1, 2, \ldots, m$,
identify $X_{j}^{s_j (n + 1)} = \bigl( X_{j}^{s_j (n)} \bigr)^{d_j (n + 1)}$,
and let
\[
P_{n, j, 1}, P_{n, j, 2}, \ldots, P_{n, j, d_j (n + 1)} \colon
  X_{j}^{s_j (n + 1)} \to X_{j}^{s_j (n)}
\]
be the resulting coordinate projection maps.
We require that,
with $t = [\mu_{j, j} (n + 1) - d_j (n + 1)] r_j (n) \in \Nz$
(which might be zero), there be a unital \hm{}
\[
\et
 \colon C \bigl( X_{j}^{s_j (n)}, \, M_{r_j (n)} \bigr)
   \to C \bigl( X_{j}^{s_j (n + 1)}, \, M_{t} \bigr)
\]
such that, for $f \in C \bigl( X_{j}^{s_j (n)}, \, M_{r_j (n)} \bigr)$,
we have
\[
\ph_{n + 1, \, n}^{j, j} (f)
 = \diag \bigl( f \circ P_{n, j, 1}, \, f \circ P_{n, j, 2}, \, \ldots,
    \, f \circ P_{n, j, d_j (n + 1)}, \, \et (f), \, 0 \bigr).
\]
\item\label{I_3Y22_GVS_Kp}
For $j = 1, 2, \ldots, m$, we define
$\kp_j \in \bigl[ 0, \frac{1}{r_j (0)} \bigr]$ by
\[
\kp_j = \limi{n} \frac{s_j (n)}{r_j (n)}.
\]
See Lemma~\ref{L_3Y23_kp_exist} below for justification.
\setcounter{TmpEnumi}{\value{enumi}}
\end{enumerate}
\end{cns}

We do not include any assumptions to force the direct limit of
such a system to be simple,
although, for our applications, that is the interesting case.

\begin{dfn}\label{D_3X30_GVS_Dfn}
A {\emph{multiseed Villadsen system}} is a direct system
which is isomorphic to a direct system of the form
described in Construction~\ref{D_3X28_GVS_df}.
The spaces $X_1, X_2, \ldots, X_m$ are called the {\emph{seed spaces}}.
\end{dfn}

\begin{lem}\label{L_3Y23_kp_exist}
Assume parts (\ref{I_3X28_GVS_m}), (\ref{I_3Y23_GVS_Amp}),
(\ref{I_3X28_GVS_Sizes}), and (\ref{I_3X30_GVS_Cumulative})
of Construction~\ref{D_3X28_GVS_df},
and let the notation be as there.
Then for $k = 1, 2, \ldots, m$, the sequence
$\bigl( \frac{s_k (n)}{r_k (n)} \bigr)_{n \in \Nz}$
is nonincreasing, and the limit
\[
\kp_k = \limi{n} \frac{s_k (n)}{r_k (n)}
\]
exists and is in $\bigl[ 0, \frac{1}{r_k (0)} \bigr]$.
\end{lem}

\begin{proof}
For $n \in \Nz$,
we have
\[
r_{k} (n + 1)
 = \sum_{j = 1}^m \mu_{k, j} (n + 1) r_j (n)
 \geq \mu_{k, k} (n + 1) r_k (n)
 \geq d_{k} (n + 1) r_k (n).
\]
Therefore
\[
\frac{s_k (n + 1)}{r_k (n + 1)}
   \leq \frac{s_k (n + 1)}{d_{k} (n + 1) r_k (n)}
   = \frac{s_k (n)}{r_k (n)}.
\]
So, for $n \in \Nz$,
\[
0 < \frac{s_k (n)}{r_k (n)} \leq \frac{s_k (0)}{r_k (0)} = \frac{1}{r_k (0)}.
\]
Thus, the sequence is nonincreasing and nonnegative,
and it is immediate that it converges
to a limit in $\bigl[ 0, \frac{1}{r_k (0)} \bigr]$.
\end{proof}

\begin{lem}\label{L_3Y26_Simple}
Assume parts (\ref{I_3X28_GVS_m}), (\ref{I_3X28_GVS_Spaces}),
(\ref{I_3Y23_GVS_Amp}), (\ref{I_3X28_GVS_Sizes}),
(\ref{I_3X30_GVS_Cumulative}), and (\ref{I_3X28_GVS_algs})
of Construction~\ref{D_3X28_GVS_df}.
Assume further
that for $n \in \N$ and $j, k \in \{ 1, 2, \ldots, m \}$
we have $\mu_{k, j} (n) \geq 1$,
and that for $j \in \{ 1, 2, \ldots, m \}$
we have $\mu_{j, j} (n) \geq d_j (n)$.
If $m = 1$, then assume $\mu_{1, 1} (n) \geq d_1 (n) + 1$.
Then the condition
in Construction \ref{D_3X28_GVS_df}(\ref{I_3X30_GVS_Growth}) is satisfied.
Further, there exist diagonal unital \hm{s}
$\ph_{n + 1, \, n} \colon A_n \to A_{n + 1}$
(as in Construction \ref{D_3X28_GVS_df}(\ref{I_3X28_GVS_maps}))
such that the conditions in
parts (\ref{I_3X30_GVS_Moremaps})--(\ref{I_3X30_GVS_Vlds})
of Construction~\ref{D_3X28_GVS_df} hold,
and such that $\dirlim_n A_n$ is simple.
\end{lem}

\begin{proof}
The hypotheses automatically imply $r_j (n) \geq m^n$
($r_1 (n) \geq 2^n$ if $m = 1$), so the condition in
Construction \ref{D_3X28_GVS_df}(\ref{I_3X30_GVS_Growth}) is satisfied.

Let $j \in \{ 1, 2, \ldots, m \}$.
Choose points $x_{j,n} \in X_{j}^{s_j (n)}$, for $n \in \N$,
such that, if for $l \geq n$ we set
\[
\begin{split}
S_{j, l, n}
& = \Bigl\{
 \bigl( P_{n, j, \nu_1} \circ P_{n+1, j, \nu_2} \circ
    \cdots \circ P_{l - 1, j, \nu_{l - n}} \bigr) (x_{j, l}) \mid
\\
& \hspace*{3em} {\mbox{}}
 {\mbox{$\nu_t \in \{ 1, 2, \ldots, d_j (n + t) \}$
             for $t = 1, 2, \ldots, l - n$}} \Bigr\},
\end{split}
\]
then, for every $n \in \Nz$, the set
$S_{j, \I, n} = \bigcup_{l = n + 1}^{\I} S_{j, l, n}$
is dense in $X_j^{s_j (n)}$.
Define $\sm \colon \{ 1, 2, \ldots, m \} \to \{ 1, 2, \ldots, m \}$
by $\sm (j) = j + 1$ when $j < m$ and $\sm (m) = 1$.
(If $m = 1$, just take $\sm (1) = 1$.)
We claim that we can choose the maps of the system so that,
for all $n \in \Nz$,
at least one of the diagonal components of $\ph_{n + 1, n}^{\sm (j), j}$
is a point evaluation at $x_{j, n}$.
If $m \neq 1$, then $\sm (j) \neq j$,
so the claim follows from $\mu_{\sm (j), j} (n + 1) \geq 1$.
If $m = 1$, so $j = 1$,
instead use $\mu_{1, 1} (n + 1) \geq d_1 (n + 1) + 1$.
Thus, the claim holds.

Make choices as in the claim.
We use \cite[Proposition 2.1]{DGNP} to prove simplicity.
Given $j \in \{ 1, 2, \ldots, m \}$, $n \in \Nz$,
and a nonempty open set $U \S X_j^{s_j (n)}$,
we must find $l_0 \geq n$ such that, for all $l \geq l_0$,
all $k \in \{ 1, 2, \ldots, m \}$, and all $x \in X_k^{s_k (l)}$, the map
\[
\ev_x \circ \ph_{l, n}^{k, j} \colon
 C \bigl( X_{j}^{s_j (n)}, M_{r_j (n)} \bigr) \to M_{r_k (l)},
\]
when written as a direct sum
of point evaluations at points in $X_{j}^{s_j (n)}$,
includes $\ev_z$ for at least one point $z \in U$.
To do this, first choose $l_0$ so that $S_{j, l_0 - 2, n} \cap U \neq \E$.
Choose $z \in S_{j, l_0 - 2, n} \cap U$.
There are $\nu_t \in \{ 1, 2, \ldots, d_j (n + t) \}$,
for $t = 1, 2, \ldots, l_0 - n - 2$, such that
\[
z = \bigl( P_{n, j, \nu_1} \circ P_{n+1, j, \nu_2} \circ
    \cdots \circ P_{l_0 - 3, j, \nu_{l_0 - n - 2}} \bigr) (x_{j, l_0 - 2}).
\]
The map $\ph_{l_0 - 1, n}^{\sm (j), j}$ has as a summand the map
$\ps = \ph_{l_0 - 1, l_0 - 2}^{\sm (j), j} \circ \ph_{l_0 - 2, n}^{j, j}$.
The map $\ph_{l_0 - 2, n}^{j, j}$ has as a summand the map
\[
\rh (f)
 = f \circ \bigl( P_{n, j, \nu_1} \circ P_{n+1, j, \nu_2} \circ
    \cdots \circ P_{l_0 - 3, j, \nu_{l_0 - n - 2}})
\]
(just $\rh = \id_{ C (X_j^{s_j (n)}, \, M_{r_j (n)} )}$ if $l_0 = n + 2$).
So $\ps$ has as a summand the \hm{}
which sends $f \in C (X_j^{s_j (n)}, \, M_{r_j (n)} )$
to the constant function $\ld (f)$ given by
\[
\ld (f) (y) = \rh (f) (x_{j, l_0 - 2}) = f (z)
\]
for $y \in X_{\sm (j)}^{s_{\sm (j)} (l_0 - 1)}$.

Now let $l \geq l_0$, let $k \in \{ 1, 2, \ldots, m \}$,
and let $x \in X_k^{s_k (l)}$.
Since
\[
\mu_{k, \sm (j)} (l_0), \, \mu_{k, k} (l_0 + 1), \, \mu_{k, k} (l_0 + 2),
\, \ldots, \, \mu_{k, k} (l)
\]
are all nonzero, $\ph_{l, l_0 - 1}^{k, \sm (j)}$ has a direct summand
of the form $\gm (g) = g \circ h$ for some \cfn{}
$h \colon X_k^{s_k (l)} \to X_{\sm (j)}^{s_{\sm (j)} (l_0 - 1)}$.
Therefore $\ev_x \circ \ph_{l, n}^{k, j}$ has as a summand
the \hm{} sending $f \in C (X_j^{s_j (n)}, \, M_{r_j (n)} )$ to
\[
(\ev_x \circ \gm \circ \ld) (f)
 = (\gm \circ \ld) (f) (x)
 = \ld (f) (h (x))
 = f (z) \, .
\]
So \cite[Proposition 2.1]{DGNP} implies that $A$ is simple.
\end{proof}

\begin{lem}\label{L_3Y21_AFK}
Let $\bigl( (A_n)_{n \in \Nz}, (\ph_{n_2, n_1})_{n_2 \geq n_1} \bigr)$
be a direct system as in Construction~\ref{D_3X28_GVS_df},
with diagonal maps (Definition~\ref{D_3X28_Diag_sys}).
Then there are a unital direct system
$\bigl( (D_n)_{n \in \Nz}, (\ps_{n_2, n_1})_{n_2 \geq n_1} \bigr)$
with direct limit $D = \dirlim_n D_n$
and maps $\ps_{\I, n} \colon D_n \to D$,
in which $D_n = \bigoplus_{j = 1}^{m} M_{r_j (n)}$ for $n \in \Nz$,
in which for $n \in \Nz$ and $k, j \in \{ 1, 2, \ldots, m \}$,
the partial embedding multiplicity of the map
$\ps_{n + 1, n}^{k, j} \colon M_{r_{j} (n)} \to M_{r_{k} (n + 1)}$
is $\mu_{k, j} (n + 1)$,
and such that the inclusions
$\io_n^j \colon M_{r_j (n)}
 \to C \bigl( X_{j}^{s_j (n)}, \, M_{r_j (n)} \bigr)$
as constant functions induce unital \hm{s} $\io_n \colon D_n \to A_n$
and a unital \hm{} $\io \colon D \to A$.
Moreover:
\begin{enumerate}
\item\label{I_24620_Indep}
The direct system
$\bigl( (D_n)_{n \in \Nz}, (\ps_{n_2, n_1})_{n_2 \geq n_1} \bigr)$
and its direct limit~$D$
depend only on $m$, on the sequences $(r_j (n))_{n \in \Nz}$,
and on the sequences $(\mu (n))_{n \in \Nz}$ of multiplicity matrices,
not on the numbers $d_j (n)$.
\item\label{I_24620_Smp}
If the conditions of Lemma~\ref{L_3Y26_Simple} hold, then $D$ is simple.
\item\label{I_24620_Contr}
If $X_j$ is contractible for $j = 1, 2, \ldots, m$,
then $(\io_n)_{*} \colon K_{*} (D_n) \to K_{*} (A_n)$
is an order isomorphism for all $n \in \Nz$
and $\io_{*} \colon K_{*} (D) \to K_{*} (A)$ is an order isomorphism.
\end{enumerate}
\end{lem}

\begin{proof}
This is immediate from the construction.
\end{proof}

We recall the definition of a solid space from \cite[Definition 3.3]{ELN}.
(In the following, presumably solidity of $X_j$ can be weakened to
$\dim_{\Q} (X_j) = \dim (X_j)$,
by using the K\"{u}nneth formula for rational cohomology
and the methods of~\cite{EllNiu2}.)

\begin{dfn}
We say that a finite dimensional second countable compact Hausdorff space
$X$ is \emph{solid} if it contains a Euclidean ball of
dimension $\dim(X)$.
\end{dfn}

\begin{prp}\label{P_3X28_GVS_rc}
Let $\bigl( (A_n)_{n \in \Nz}, (\ph_{n_2, n_1})_{n_2 \geq n_1} \bigr)$
be a multiseed Villadsen system with diagonal maps.
Let $A = \dirlim_n A_n$.
Let the notation be as in Construction~\ref{D_3X28_GVS_df}.
Assume that the spaces $X_1, X_2, \ldots, X_m$ are \fd, second countable,
and solid.
Then
\[
\rc (A)
 = \frac{1}{2} \drr (A)
 = \frac{1}{2}
  \drr \bigl( (A_n)_{n \in \Nz}, (\ph_{n_2, n_1})_{n_2 \geq n_1} \bigr)
 = \frac{1}{2} \max_{1 \leq j \leq m (n)} \kp_j \dim (X_j).
\]
\end{prp}

We do not need simplicity of~$A$.

The hypothesis of diagonal maps is used only to be sure that the \pj{}
$\ph (p)$ in Lemma~\ref{L_3X31_ranks} is trivial.
If in the situation of Lemma~\ref{L_3X31_ranks} with not necessarily
diagonal maps, it is automatic that $\ph (p)$ is trivial,
then we do not need the maps in Proposition~\ref{P_3X28_GVS_rc}
to be diagonal except for the parts which use coordinate \pj{s},
as in Construction \ref{D_3X28_GVS_df}(\ref{I_3X30_GVS_Vlds}).

The following two facts are used often enough elsewhere
that it seems useful to record them separately.
The first is closely related to the proof of \cite[Lemma~2.13]{EllNiu2}.

\begin{lem}\label{L_24620_Cech}
Let $q \in \N$,
let $X$ be a \cms{} with $\dim (X) \leq 2 q + 1$,
and let $\et \in \chH^{2 q} (X; \Z)$ (\v{C}ech cohomology)
be an element of infinite order.
Then there are a rank~$q$ vector bundle $E$ over~$X$
and $l \in \N$ such that the total Chern class $c (E)$ of~$E$
has the property that $c (E) - (1 + l \et)$
is a torsion element of $\chH^{2 q} (X; \Z)$.
\end{lem}

We will apply this with $X = S^{2 q}$.
In this case, it is known that the possible values of $c_q (E)$
are exactly the elements of $(q - 1)! H^{2 q} (S^{2 q}; \Z)$.
We have not found this in the literature.
The closest we have found is \cite[Chapter~20, Corollary~9.8]{Husemoller},
which gives K-theory classes rather than vector bundles of specified rank.
However, it is not hard to deduce the statement above as given.

\begin{proof}[Proof of Lemma~\ref{L_24620_Cech}]
Since $X$ is an inverse limit of finite complexes of dimension
at most $2 q + 1$,
and both \v{C}ech cohomology $\chH^{*} (Y; \Z)$
and the set ${\operatorname{Vect}} (Y)$ of
\im{} classes of vector bundles over~$Y$ are \ct{} functors,
we may assume that $X$ is a finite complex,
and work with singular cohomology.
Treating connected components separately,
we may further assume that $X$ is connected.

Let $\io \colon H^* (X; \Z) \to H^* (X; \Q)$
be the change of coefficients map.
Then $\io (\et) \neq 0$.
Apply \cite[Lemma~3.2]{HrsPh3}
with $Z = X$ and $\io (\et)$ in place of~$\et$,
getting $l \in \N$ and a vector bundle $E_0$ over~$X$
such that $\io (c_q (E_0)) = l \cdot \io (\et)$ and
$\io (c_j (E)) = 0$ for $j = 1, 2, \ldots, q - 1$.
(There is a misprint in the proof: in the third paragraph,
it is $Z'$, not~$Z$, which has no cells
of dimension strictly between $0$ and~$2 q$.)

Set $q_0 = \rank (E_0)$.
Since $c_q (E_0) \neq 0$, we have $q_0 \geq q$.
Repeated use of \cite[Chapter~9, Proposition~1.1]{Husemoller}
gives a vector bundle $E$ with $\rank (E) = q$ such that
$E_0 \cong E \oplus (X \times \C^{q_0 - q})$.
We have $c_j (E) = c_j (E_0)$ for all $j \in \Nz$,
and $c_j (E) = 0$ for $j > \rank (E) = q$.
So $\io ( c (E)) = 1 + l \cdot \io (\et)$.
Thus, $c (E) - (1 + l \et) \in \Ker (\io)$, and the conclusion follows.
\end{proof}

\begin{lem}\label{L_24620_Subeq}
Let $A$ be a \ca, let $p, q \in A$ be \pj{s},
and suppose there is $v \in A$ such that $\| v p v^* - q \| < 1$.
Then $q$ is \mvnt{} a sub\pj{} of~$p$.
\end{lem}

\begin{proof}
As in the proof of \cite[Corollary~1.11]{HP_asymmetry},
we have $\| q v p v^* q - q \| < 1$,
so $q v p v^* q$ has an inverse $c \in (q A q)_{+}$.
Then $w = c^{1 / 2} q v p$ satisfies $w w^* = q$ and $w^* w \leq p$.
\end{proof}

\begin{proof}[Proof of Proposition~\ref{P_3X28_GVS_rc}]
We have $\rc (A) \leq \frac{1}{2} \drr (A)$ by Lemma~5.1 of~\cite{AsAsv}.
The inequality
\[
\frac{1}{2} \drr (A)
 \leq \frac{1}{2}
  \drr \bigl( (A_n)_{n \in \Nz}, (\ph_{n_2, n_1})_{n_2 \geq n_1} \bigr)
\]
holds by definition.

Next, solidity of $X_{j}$ implies
$\dim (X_{j}^{s_j (n)}) = s_j (n) \dim (X_{j})$,
and the limit of the maximum of finitely many nonincreasing sequences
is the maximum of their limits, so
\[
\begin{split}
& \drr \bigl( (A_n)_{n \in \Nz}, (\ph_{n_2, n_1})_{n_2 \geq n_1} \bigr)
\\
& \hspace*{3em} {\mbox{}}
  = \limsup_{n \to \I} \max_{1 \leq j \leq m}
      \frac{\dim (X_{j}^{s_j (n)})}{r_j (n)}
 = \limsup_{n \to \I} \max_{1 \leq j \leq m}
      \frac{s_j (n) \dim (X_{j})}{r_j (n)}
\\
& \hspace*{3em} {\mbox{}}
 = \max_{1 \leq j \leq m} \dim (X_{j}) \lim_{n \to \I}
      \frac{s_j (n)}{r_j (n)}
 = \max_{1 \leq j \leq m (n)} \kp_j \dim (X_j).
\end{split}
\]

It remains to prove that for $j = 1, 2, \ldots, m$ we have
$\rc (A) \geq \frac{1}{2} \kp_j \dim (X_j)$.
Fix~$j$.
We may assume $\dim (X_j) > 0$.
We assume that the system is equal to one as described in
Construction \ref{D_3X28_GVS_df},
not merely isomorphic to such a system.
Also, for any \hm{} $\ps \colon B \to C$ and any $t \in \N$,
let (temporarily) $\ps^t \colon M_t (B) \to M_t (C)$
be the corresponding \hm{} of matrix algebras.

Let $\ep > 0$.
By Construction \ref{D_3X28_GVS_df}(\ref{I_3X30_GVS_Growth})
and Lemma~\ref{L_3Y23_kp_exist}, we can choose $n_{0} \in \N$ so large that
\begin{equation}\label{Eq_3X29_Ch_ep}
r_i (n_{0}) > \frac{10}{\ep} + 1
\end{equation}
for $i = 1, 2, \ldots, m$ and
\begin{equation}\label{Eq_3X31_Ch_ep_2}
\frac{s_j (n_{0})}{r_j (n_{0})} - \kp_j < \frac{\ep}{\dim (X_j)}.
\end{equation}
Set
\begin{equation}\label{Eq_3X31_r0}
r_{0} = \min \bigl( r_1 (n_{0}), \, r_2 (n_{0}),
                \, \ldots, \, r_m (n_{0}) \bigr) - 1.
\end{equation}
Choose $\nu_1, \nu_2, \ldots, \nu_m, g \in \N$ such that
\begin{equation}\label{Eq_3X29_t_nu}
\left( \kp_j \dim (X_j) - \frac{5}{r_{0}} \right) r_i (n_{0})
 \leq \nu_i
 < \left( \kp_j \dim (X_j) - \frac{4}{r_{0}} \right) r_i (n_{0})
\end{equation}
for $i = 1, 2, \ldots, m$, and
\begin{equation}\label{Eq_3X29_g}
s_j (n_{0}) \dim (X_j) - 2 < 2 g + 1 \leq s_j (n_{0}) \dim (X_j).
\end{equation}

Since $X_j^{s_j (n_{0})}$ contains a closed ball
of dimension $s_j (n_{0}) \dim (X_j)$,
there is an injective \ct{} map $h \colon S^{2 g} \to X_j^{s_j (n_{0})}$.
Set $Z = h (S^{2 g})$.

Let $\et$ be a generator of $H^{2 g} (S^{2 g}; \Z)$.
By Lemma~\ref{L_24620_Cech}, and because $H^{*} (S^{2 g}; \Z)$
has no torsion,
there are $l \in \N$ and a rank~$g$ vector bundle $E$ over~$S^{2 g}$
such that the total Chern class of~$E$ is given by $c (E) = 1 + l \et$.
Choose $N \in \N$
and a \pj{} $p \in M_{N} (C (S^{2 g}, \, M_{r_j (n_{0})}))$
whose range is isomorphic to~$E$;
also require $N \geq \nu_i$ for $i = 1, 2, \ldots, m$.
Define
\[
\af \colon C (X_j^{s_j (n_{0})}, \, M_{r_j (n_{0})})
 \to C (S^{2 g}, \, M_{r_j (n_{0})})
\]
by $\af (f) = f \circ h$ for $f \in C (X_j^{s_j (n_{0})}, M_{r_j (n_{0})})$.

Define $q_{0} \in M_{N} (C (Z))$ by $q_{0} = p \circ h^{-1}$.
Then $q_{0}$ is a \pj{} with constant rank~$g$.
Standard functional calculus arguments show that
there are a compact \nbhd{} $Y$ of~$Z$
and a \pj{} $q \in C (Y)$ with constant rank~$g$ such that $q |_{Z} = q_{0}$.
Choose a \cfn{} $\ld \colon X_j^{s_j (n_{0})} \to [0, 1]$
such that $\ld = 1$ on~$Z$
and $\supp (\ld) \S \sint (Y)$.
Define $a \in M_{N} (A_{n_{0}})$ by
\begin{equation}\label{Eq_3X31_aj}
a_i (x) = \begin{cases}
   \ld (x) q (x) & \hspace*{1em} {\mbox{$i = j$ and $x \in Y$}}
       \\
   0           & \hspace*{1em} {\mbox{otherwise}}.
\end{cases}
\end{equation}
For $i = 1, 2, \ldots, m$,
take $b_i \in M_N \bigl( C (X_i^{s_i (n_{0})}, \, M_{r_i (n_{0}} ) \bigr)$
to be a constant \pj{} of rank~$\nu_i$,
and set $b = (b_1, b_2, \ldots, b_n) \in A_{n_{0}}$.

We claim that for all $n \geq n_{0}$ there is no $v \in M_{N} (A_n)$
such that
(recalling that $(\ph_{n, n_{0}})^N$ is the matrix amplification)
\[
\left\| v (\ph_{n, n_{0}})^N (b) v^* - (\ph_{n, n_{0}})^N (a) \right\|
 < \frac{1}{2} \, .
\]
To prove the claim,
suppose that $n \geq n_{0}$ and such an element~$v$ exists.
Write $v = (v_1, v_2, \ldots, v_m)$.
Restricting to the $j$-th summand, we get
\begin{equation}\label{Eq_3X29_less_half}
\bigl\| v_j (\ph_{n, n_{0}})^N (b)_j v_j^*
     - (\ph_{n, n_{0}})^N (a)_j \bigr\| < \frac{1}{2}.
\end{equation}

By iteration of Construction \ref{D_3X28_GVS_df}(\ref{I_3X30_GVS_Vlds}),
there are an identification
\begin{equation}\label{Eq_3X29_IdPrd}
X_j^{s_j (n)} = \bigl( X_j^{s_j (n_{0})} \bigr)^{s_j (n) / s_j (n_{0})}
\end{equation}
with coordinate \pj{s}
\[
Q_1, Q_2, \ldots, Q_{s_j (n) / s_j (n_{0})} \colon
 \bigl( X_j^{s_j (n_{0})} \bigr)^{s_j (n) / s_j (n_{0})}
  \to X_j^{s_j (n_{0})},
\]
a constant \pj{} $e \in C (X_j^{s_j (n)}, \, M_{r_j (n)} )$
of rank $r_j (n_{0}) s_j (n) / s_j (n_{0})$,
and a constant unitary $u \in C (X_j^{s_j (n)}, \, M_{r_j (n)} )$,
such that for $f = (f_1, f_2, \ldots, f_m) \in A_{n_{0}}$, we have
\[
u e [ \ph_{n, \, n_{0}} (f)_j ] e u^*
 = \diag \bigl( f_j \circ Q_{1}, \, f_j \circ Q_{2}, \, \ldots,
    \, f_j \circ Q_{s_j (n) / s_j (n_{0})}, \, 0 \bigr).
\]
Using~(\ref{Eq_3X29_IdPrd}) and the
canonical identification of $C (Y_1) \otimes C (Y_2)$ with
$C (Y_1 \times Y_2)$, set
\[
\bt = \af^{\otimes s_j (n) / s_j (n_{0})} \colon
   C (X_j^{s_j (n)}) \to C \bigl( (S^{2 g})^{s_j (n) / s_j (n_{0})} \bigr).
\]
Also let
\[
R_1, R_2, \ldots, R_{s_j (n) / s_j (n_{0})} \colon
 (S^{2 g})^{s_j (n) / s_j (n_{0})} \to S^{2 g}
\]
be the coordinate \pj{s}.
Then
\begin{equation}\label{Eq_3X30_f_dfn}
f = \bt^{N r_j (n)}
     \bigl( (1_{M_{N}} \otimes u e)
        (\ph_{n, \, n_{0}})^{N} (a)_j
        (1_{M_{N}} \otimes e u^*) \bigr)
\end{equation}
is a \pj{} in
$C \bigl( (S^{2 g})^{s_j (n) / s_j (n_{0})}, \, M_{N r_j (n)} \bigr)$
of rank $g s_j (n) / s_j (n_{0})$
and whose range is isomorphic to the vector bundle
\[
F = \bigoplus_{i = 1}^{s_j (n) / s_j (n_{0})} R_i^* (E)
  = E^{\times s_j (n) / s_j (n_{0})},
\]
the set theoretic Cartesian product
of $s_j (n) / s_j (n_{0})$ copies of~$E$.

Since $\ph_{n, \, n_{0}}$ is unital, we can use~(\ref{Eq_3X29_t_nu})
and Lemma~\ref{L_3X31_ranks} to see that
the \pj{} $(\ph_{n, \, n_{0}})^{N} (b)_j$ is a constant \pj{}
in $C \bigl( X_j^{s_j (n)}, \, M_{N r_j (n)} ) \bigr)$
whose rank $\nu$ satisfies
\begin{equation}\label{Eq_3X31_nu}
\left( \kp_j \dim (X_j) - \frac{5}{r_{0}} \right) r_j (n)
 \leq \nu
 < \left( \kp_j \dim (X_j) - \frac{4}{r_{0}} \right) r_j (n).
\end{equation}
Moreover,
$\bt^{N r_j (n)} \bigl( (\ph_{n, \, n_{0}})^{N} (b)_j \bigr)$
is a constant \pj{}
\[
z \in C \bigl( (S^{2 g})^{s_j (n) / s_j (n_{0})}, \, M_{N r_j (n)} ) \bigr)
\]
with rank~$\nu$.
{}From~(\ref{Eq_3X29_less_half}) we get, with $y = 1_{M_{N}} \otimes u e$,
\[
\bigl\| y v_j (\ph_{n, \, n_{0}})^{N} (b)_j v_j^* y^*
    - y (\ph_{n, \, n_{0}})^{N} (a)_j y^* \bigr\| < \frac{1}{2}.
\]
Therefore, with $w = \bt^{N r_j (n)} (y v_j )$,
and using~(\ref{Eq_3X30_f_dfn}), we have
\[
\| w z w^* - f \| < \frac{1}{2}.
\]
Lemma~\ref{L_24620_Subeq} now implies that
$F$ is isomorphic to a subbundle of the rank~$\nu$ trivial bundle.

We now show that
\begin{equation}\label{Eq_3X29_frac_ineq}
\nu \geq \frac{2 g s_j (n)}{s_j (n_{0})}.
\end{equation}
The proof of this is similar to the proof of
\cite[Lemma~1.9]{HP_asymmetry}.
(Several references in that proof to ``Chern character'' should be to
Chern classes.)
To simplify notation, set $t = s_j (n) / s_j (n_{0})$.
By the same reasoning as there,
\[
c (F) = \prod_{i = 1}^t ( 1 + l R_i^* (\et)).
\]
If $F \oplus V \cong (S^{2 g})^t \times \C^{\nu}$, then
\[
c (V) = c (F)^{-1} = \prod_{i = 1}^t ( 1 - l R_i^* (\et)).
\]
In particular,
\[
c_{g t} (V) = (-1)^t l^t \prod_{i = 1}^t R_i^* (\et)
\]
is nonzero, so that $\rank (V) \geq g t$.
This implies~(\ref{Eq_3X29_frac_ineq}), as wanted.

Using (\ref{Eq_3X31_nu}) at the first step,
using (\ref{Eq_3X31_r0}) at the second step,
using (\ref{Eq_3X29_g}) at the fourth step,
and using the inequalities
\[
\kp_j \leq \frac{s_j (n_{0})}{r_j (n_{0})}
\andeqn
\kp_j \leq \frac{s_j (n)}{r_j (n)}
\]
at the third and the fifth steps,
we get
\[
\begin{split}
\left( \frac{r_j (n_{0})}{r_j (n)} \right) \nu
& < \frac{r_j (n_{0})}{r_j (n)}
      \left( \kp_j \dim (X_j) - \frac{4}{r_{0}} \right) r_j (n)
  < \kp_j r_j (n_{0}) \dim (X_j) - 4
\\
& \leq \left( \frac{r_j (n_{0})}{s_j (n_{0})} \right) \kp_j
     \bigl[ s_j (n_{0}) \dim (X_j) - 3 \bigr] - 1
\\
& < \left( \frac{r_j (n_{0})}{s_j (n_{0})} \right) \kp_j \cdot 2 g - 1
  \leq \left( \frac{r_j (n_{0})}{s_j (n_{0})} \right)
         \left( \frac{s_j (n)}{r_j (n)} \right) \cdot 2 g - 1.
\end{split}
\]
Therefore
\[
\nu < \frac{2 g s_j (n)}{s_j (n_{0})} - \frac{r_j (n)}{r_j (n_{0})}
   \leq \frac{2 g s_j (n)}{s_j (n_{0})} - 1,
\]
contradicting~(\ref{Eq_3X29_frac_ineq}).
The claim is proved.

The claim implies that
$(\ph_{\I, n_0})^{N} (a) \not\precsim_A (\ph_{\I, n_0})^{N} (b)$.

We claim that for all $\ta \in \T (A)$, we have
\begin{equation}\label{Eq_3X31_for_rc}
d_{\ta} \bigl( (\ph_{\I, n_0})^{N} (a) \bigr)
   + \frac{1}{2} \kp_j \dim (X_j) - \ep
 < d_{\ta} \bigl( (\ph_{\I, n_0})^{N} (b) \bigr).
\end{equation}
To prove the claim, it clearly suffices
to prove that for every $\sm \in \T (A_{n_0})$, we have
\[
d_{\sm} (a) + \frac{1}{2} \kp_j \dim (X_j) - \ep < d_{\sm} (b).
\]
For $i = 1, 2, \ldots, m$, using (\ref{Eq_3X29_t_nu})
at the second step
and (\ref{Eq_3X29_Ch_ep}) at the third step, we have
\[
\frac{\rank (b_i)}{r_i (n_{0})}
 = \frac{\nu_i}{r_i (n_{0})}
 \geq \kp_j \dim (X_j) - \frac{5}{r_{0}}
 > \kp_j \dim (X_j) - \frac{\ep}{2}.
\]
Therefore, for any $\sm \in \T (A_{n_{0}})$,
\[
d_{\sm} (b) \geq \kp_j \dim (X_j) - \frac{\ep}{2}.
\]
On the other hand,
using (\ref{Eq_3X31_aj}) at the first step,
(\ref{Eq_3X29_g}) at the second step,
and (\ref{Eq_3X31_Ch_ep_2}) at the fourth step,
\[
\begin{split}
d_{\sm} (a)
& \leq \frac{g}{r_j (n_{0})}
  \leq \frac{1}{2 r_j (n_{0})} \bigl[ s_j (n_{0}) \dim (X_j) - 1 \bigr]
\\
& \leq \frac{1}{2} \left( \frac{s_j (n_{0})}{r_j (n_{0})} \right)
            \dim (X_j)
  \leq \frac{1}{2} \kp_j \dim (X_j) + \frac{\ep}{2}.
\end{split}
\]
Therefore
\[
d_{\sm} (b) - d_{\sm} (a)
 > \frac{1}{2} \kp_j \dim (X_j) - \ep.
\]
The claim is proved.

The claim implies that $\rc (A) > \frac{1}{2} \kp_j \dim (X_j) - \ep$.
Since $\ep > 0$ is arbitary,
we get $\rc (A) > \frac{1}{2} \kp_j \dim (X_j)$,
as desired.
\end{proof}

\section{K-theory and traces}\label{Sec_K_theory}

We compute the ordered K-theory
in a special case of Construction~\ref{D_3X28_GVS_df},
with $m = 2$ and where $X_1$ and $X_2$ are contractible,
along with some additional restrictions on the parameters
intended to simplify the computations.
There is no inherent difficulty in carrying out a similar computation
with larger~$m$, but this is not needed for our results.
We describe the tracial state space and its pairing with K-theory
in a somewhat different special case.
Our examples will be covered by both these cases.

\begin{lem}\label{L_3Y21_Ord_Kth}
Let $\bigl ( (A_n)_{n \in \Nz}, (\ph_{n_2, n_1})_{n_2 \geq n_1} \bigr)$
be a direct system as in Construction~\ref{D_3X28_GVS_df},
satisfying the following additional conditions:
\begin{enumerate}
\item\label{I_L_3Y21_Ord_Kth_2}
$m = 2$ and $r_1 (0) = r_2 (0) = 1$.
\item\label{I_L_3Y21_Ord_Kth_Ct}
$X_1$ and $X_2$ are contractible.
\item\label{I_L_3Y21_Ord_Kth_Diag}
All maps are diagonal.
\item\label{I_L_3Y21_Ord_Kth_Eq}
There are sequences $(c (n))_{n \in \N}$ and $(l (n))_{n \in \N}$ in~$\N$
such that, for every $n \in \Nz$,
\[
\mu_{1, 1} (n) = \mu_{2, 2} (n) = l (n) - c (n)
\andeqn
\mu_{1, 2} (n) = \mu_{2, 1} (n) = c (n).
\]
\item\label{I_L_4624_Ord_Kth_l2c1}
$l (n) \geq 2 c (n) + 1$ for all $n \in \N$.
\end{enumerate}
For $n \in \Nz$, let
\[
r (n) = \prod_{k = 1}^{n} l (k)
\andeqn
r' (n) = \prod_{k = 1}^n [ l (k) - 2 c (k) ].
\]
Then the numbers $r_1 (n)$ and $r_2 (n)$ of
Construction \ref{D_3X28_GVS_df}(\ref{I_3X30_GVS_Cumulative})
satisfy $r_1 (n) = r_2 (n) = r (n)$ for all $n \in \Nz$.
As in Lemma~\ref{L_3Y21_AFK},
let $D_n = M_{r (n)} \oplus M_{r (n)}$ be the subalgebra of~$A_n$
consisting of the constant functions in the direct summands of~$A_n$.
Then there is a unique \hm{} $\gm \colon K_0 (A) \to \Q \oplus \Q$
such that, following the notation of Lemma~\ref{L_3Y21_AFK},
and identifying $K_0 (D_n) = \Z \oplus \Z$ in the standard way,
for all $n \in \Nz$ and $m_1, m_2 \in \Z$ we have
\[
[\gm \circ \io_{*} \circ (\ps_{\I, n})_{*} ] (m_1, m_2)
 = \left( \frac{m_1 + m_2}{2 r (n)}, \, \frac{m_1 - m_2}{2 r' (n)} \right).
\]
Moreover, $\gm$ is injective.
\end{lem}

We will use the notation and maps in the proof in several later proofs.

The system
$\bigl ( (A_n)_{n \in \Nz}, (\ph_{n_2, n_1})_{n_2 \geq n_1} \bigr)$
in Lemma~\ref{L_3Y21_Ord_Kth}
has the form shown in the following diagram, where,
for simplicity, we do not indicate the connecting maps:
\begin{equation}
\xymatrix{
C (X_1) \otimes M_{r (0)} \ar[dr]  \ar[r]
& C (X_1^{s_1 (1)})\otimes M_{r (1)}\ar[dr]
\ar[r]
& C (X_1^{s_1 (2)})\otimes M_{r (2)} \ar[dr]
\ar[r] & \cdots
\\
C (X_2) \otimes M_{r (0)} \ar[ur]  \ar[r]
 & C (X_2^{s_2 (1)}) \otimes M_{r (1)} \ar[ur]
\ar[r]
& C (X_2^{s_2 (2)}) \otimes M_{r (2)} \ar[ur]
\ar[r]  & \cdots .}
\end{equation}
The AF~system
$\bigl ( (D_n)_{n \in \Nz}, (\ps_{n_2, n_1})_{n_2 \geq n_1} \bigr)$ from
Lemma~\ref{L_3Y21_AFK} has the following Bratteli diagram, in which the
labels indicate the multiplicities of the maps:
\begin{equation}
\xymatrix{
M_{r (0)} \ar[drr]^(.4){c (1)}  \ar[rr]^{l (1) - c (1)}
& &  M_{r (1)}\ar[drr]^(.4){c (2)}
\ar[rr]^{l (2) - c (2)}
& & M_{r (2)} \ar[drr]^(.4){c (3)}
\ar[rr]^{l (3) - c (3)}  && \cdots
\\
M_{r (0)} \ar[urr]_(.7){c (1)}  \ar[rr]_{l (1) - c (1)}
& &  M_{r (1)} \ar[urr]_(.7){c (2)}
\ar[rr]_{l (2) - c (2)}
& &   M_{r (2)} \ar[urr]_(.7){c (3)}
\ar[rr]_{l (3) - c (3)}  && \cdots .}
\end{equation}

\begin{proof}[Proof of Lemma~\ref{L_3Y21_Ord_Kth}]
It is immediate that $r_1 (n) = r_2 (n) = r (n)$ for all $n \in \Nz$.

We follow the notation of Lemma~\ref{L_3Y21_AFK}.
In particular, $D = \dirlim_n D_n$,
$\io_n \colon D_n \to A_n$ is the inclusion,
and $\io \colon D \to A$ is the induced map on direct limits.
Since $(\io_n)_{*}$ and $\io_{*}$ are isomorphisms
(by Lemma~\ref{L_3Y21_AFK}), it suffices to prove the statement
of Lemma~\ref{L_3Y21_Ord_Kth} for $D$ in place of~$A$.

For $n \in \N$ set $G_n = \Z \oplus \Z$.
Let $\nu_n \colon G_n \to K_0 (D_n)$ be the unique group \hm{}
which sends $(1, 0)$ to the class of a rank one \pj{}
in the first direct summand in $D_n = M_{r (n)} \oplus M_{r (n)}$
and $(0, 1)$ to the class of a rank one \pj{}
in the second direct summand.
Then $\nu_n$ is an isomorphism.
Let $\delta_{n + 1, n} \colon G_n \to G_{n + 1}$
be given by the integer matrix
\[
\left(
\begin{matrix}
l (n + 1) - c (n + 1) & c (n + 1) \\
c (n + 1) & l (n + 1) - c (n + 1)
\end{matrix}
\right),
\]
so that
$\nu_{n + 1} \circ \delta_{n + 1, n} = (\ps_{n + 1, n})_* \circ \nu_n$.
Let $G = \dirlim_n G_n$ be the direct limit of the resulting system,
with standard maps $\delta_{\infty, n} \colon G_n \to G$.
Thus, there is an isomorphism $\nu \colon G \to K_0 (D)$.

The hypotheses imply $r' (n) > 0$ for all $n \in \Nz$.
For $n \in \N$, define $\gamma_n \colon G_n \to \Q \oplus \Q$ by
\[
\gamma_n (m_1, m_2)
 = \left( \frac{m_1 + m_2}{2 r (n)}, \, \frac{m_1 - m_2}{2 r' (n)} \right) .
\]
We claim that $\gamma_{n + 1} \circ \delta_{n + 1, n} = \gamma_n$.
Indeed, for $m_1, m_2 \in \Z$, we have
\[
\begin{split}
(\gamma_{n + 1} \circ \delta_{n + 1, n}) (m_1, m_2)
& =  \gamma_{n + 1} \Bigl( [ l (n + 1) - c (n + 1) ] m_1 + c (n + 1)m_2,
\\
& \hspace*{5em} {\mbox{}}
  c (n + 1)m_1 + [ l (n + 1)-c (n + 1) ] m_2 \Bigr)
\\
& = \left( \frac{l (n + 1)(m_1 + m_2)}{2 r (n + 1)},
      \, \, \frac{ [l (n + 1) - 2 c (n + 1)]
              (m_1 - m_2)}{2 r' (n + 1)} \right)
\\
& = \gamma_n (m_1, m_2) \, ,
\end{split}
\]
as claimed.
We thus have a homomorphism $\gamma \colon G \to \Q \oplus \Q$
satisfying $\gamma \circ \delta_{\infty, n} = \gamma_n$ for all $n \in \N$.
Because the maps $\gamma_n$ are all injective,
the map $\gamma$ is injective as well.
\end{proof}

We now want to identify the range of
the map $\gamma$ in Lemma~\ref{L_3Y21_Ord_Kth}.
This is straightforward, but for convenience,
we add another assumption to guarantee that it surjective.
We do this only to slightly simplify the argument and notation.
In the general case,
it is not hard to describe the range of $\gamma$
in terms of the supernatural numbers $\prod_{n = 1}^{\infty} l (n)$
and $\prod_{n = 1}^{\infty} [l (n) - 2 c (n)]$,
but do not need this.

\begin{lem}\label{lem-range-gamma}
Let the hypotheses and notation be as in Lemma~\ref{L_3Y21_Ord_Kth}.
Further assume that for any $k \in \N$ there exist $n, n' \in \N$
such that $k | r (n)$ and $k | r' (n')$.
Then $\gamma$ is surjective.
\end{lem}

\begin{proof}
Using the notation of the proof of Lemma~\ref{L_3Y21_Ord_Kth},
we have
\[
\gamma_n (1, 1) = \left( \frac{1}{r (n)}, \, 0 \right)
\andeqn
\gamma_n (1, -1) = \left( 0, \, \frac{1}{r' (n)} \right).
\]
Thus, the image of $\gamma_n$ contains
$\Z \bigl[ \frac{1}{r (n)} \bigr] \oplus \Z \bigl[ \frac{1}{r' (n)} \bigr]$.
The assumption in the lemma now immediately implies that
the union of the images of $\gamma_n$ for $n \in \N$ contains $\Q \oplus \Q$.
\end{proof}

\begin{rmk}\label{R_3Y14_Force}
One can ensure that the the conditions of Lemma~\ref{lem-range-gamma}
hold by setting $c (n) = 1$ for all $n$
and choosing the number $l (n)$
so that $n | l(2n)$ and $n | [l (2 n + 1) - 2]$ for all~$n$.
This is easy to do while satisfying the other constraints,
as they only involve the growth rate of the
sequence $(l (n))_{n \in \Nz}$.
Allowing other values of $c (n)$ only adds flexibility.
\end{rmk}

We now describe the order structure of $K_0 (D)$.
Here is a picture of the positive cone,
with $\ld$ as in Lemma~\ref{lem-order-structure-K_0};
the scales on the axes are not the same:

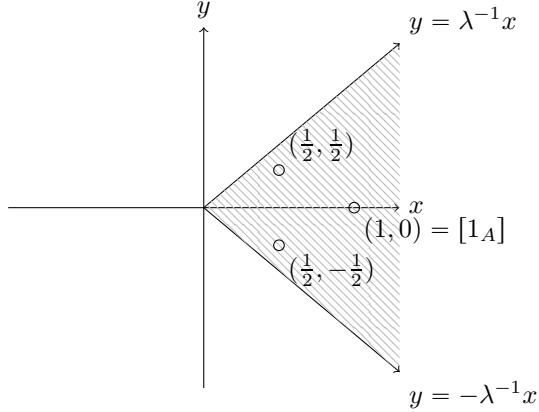
\begin{center}
\begin{tikzpicture}[scale=2]
\draw[->] (-1.3,0) -- (1.3,0) node[right] {$x$};
\draw[->] (0,-1.2) -- (0,1.2) node[above] {$y$};

\draw[->] (0,0) -- (40:1.7) node[above right] {$y = \ld^{-1} x$};
\draw[->] (0,0) -- (-40:1.7) node[below right] {$y = - \ld^{-1} x$};

\fill[pattern=north west lines,
 pattern color=gray!50] (40:1.7) -- (0,0) -- (-40:1.7) -- cycle;

\draw (1,0) circle (1pt) node[below right] {$(1,0) = [1_A]$};
\draw( 0.5,0.25) circle (1pt) node[above right] {$(\frac{1}{2},\frac{1}{2})$};
\draw (0.5,-0.25) circle (1pt) node[below right]
   {$(\frac{1}{2},-\frac{1}{2})$};
\end{tikzpicture}
\end{center}

\begin{lem}\label{lem-order-structure-K_0}
Let the hypotheses be as in Lemma~\ref{L_3Y21_Ord_Kth},
assume that the condition of Lemma~\ref{lem-range-gamma} is satisfied,
and assume that $c (n) > 0$ for infinitely many~$n$.
Define
\[
\lambda = \prod_{n = 1}^{\infty} \left( 1 - \frac{2 c (n)}{l (n)} \right) .
\]
Equip $\Q \oplus \Q$ with the order in which
$(x, y) \in \Q \oplus \Q $ is positive
if and only if $(x, y) = (0, 0)$ or $x > \lambda |y|$,
and with the order unit $(1, 0)$.
Then the map $\gamma \colon K_0 (A) \to \Q \oplus \Q$
of Lemma~\ref{L_3Y21_Ord_Kth}
is an isomorphism of scaled ordered groups.
\end{lem}

\begin{proof}
We continue to use the notation
of the proof of Lemma~\ref{L_3Y21_Ord_Kth}.
In particular, $G_n = \Z \oplus \Z$,
$\gamma_n \colon G_n \to \Q \oplus \Q$ is
\[
\gamma_n (m_1, m_2)
 = \left( \frac{m_1 + m_2}{2 r (n)}, \, \frac{m_1 - m_2}{2 r' (n)} \right)
   \, ,
\]
$D_n = M_{r (n)} \oplus M_{r (n)}$,
and $\nu_n \colon G_n \to K_0 (D_n)$ is the obvious isomorphism.
For $n \in \Nz$, we equip $G_n = \Z \oplus \Z$
with the positive cone $(\Nz)^2 \S \Z \oplus \Z$ and order unit
$(r (n), r (n))$,
so that $\nu_n$ is an isomorphism of scaled ordered groups.
Then the isomorphism $\nu \colon K_0 (D) \to K_0 (A)$
in the proof of Lemma~\ref{L_3Y21_Ord_Kth}
is an isomorphism of scaled ordered groups.

That $\gm ([1_A]) = (1, 0)$ is immediate.
We must show that $(x, y) \in \Q \oplus \Q$
is positive \ifo{} there are $n \in \Nz$ and a positive element
$(m_1, m_2) \in \Z \oplus \Z$ such that $\gamma_n (m_1, m_2) = (x, y)$.

Assume that $\gamma_n (m_1, m_2) = (x, y)$ with $m_1, m_2 \geq 0$.
Then
\begin{equation}\label{eqn-lem-order-structure-x}
x = \left( \prod_{k = 1}^{n} l (k) \right)^{-1} (m_1 + m_2)
\end{equation}
and
\begin{equation}\label{eqn-lem-order-structure-y}
y = \left( \prod_{k = 1}^{n} [ l (k) - 2 c (k) ] \right)^{-1} (m_1 - m_2)
 \, .
\end{equation}
Since $m_1 + m_2 \geq | m_1 - m_2 |$, we have
\[
x \geq \prod_{k = 1}^n\left( 1 - \frac{2 c (k)}{l (k)} \right) |y|
 \geq \lambda |y| \, .
\]
If $x = 0$ then $y = 0$.
If $x > 0$ and $y = 0$ then $x > \lambda |y|$ trivially.
If $y \neq 0$ then $x > \lambda |y|$ follows from the fact that
$\prod_{k = 1}^n\left( 1 - \frac{2 c (k)}{l (k)} \right) > \lambda$.

For the converse, suppose $(x, y) \in \Q \oplus \Q$ is positive.
We need to find $n \in \Nz$ and $m_1, m_2 \in \Nz$
such that $\gamma_n (m_1, m_2) = (x, y)$.
This is certainly true if $(x, y) = (0, 0)$.
So assume that $x > \lambda |y|$.

Because
\[
\lim_{n \to \infty} \prod_{k = 1}^n \left( 1 - \frac{2 c (k)}{l (k)} \right)
 = \lambda,
\]
there exists $n \in \Nz$ such that
\begin{equation}\label{Eq_240620_x_ge}
x > \prod_{k = 1}^n\left( 1 - \frac{2 c (k)}{l (k)} \right) |y| \, .
\end{equation}
By Lemma~\ref{lem-range-gamma}, increasing $n$ if necessary,
we may assume that $(x, y)$ is in the image of~$\gm_{n}$.
The relation~(\ref{Eq_240620_x_ge}) still holds.
Let $(m_1, m_2) \in \Z \oplus \Z$
satisfy $\gm_{n} (m_1, m_2) = (x, y)$.
Using equations (\ref{eqn-lem-order-structure-x})
and~(\ref{eqn-lem-order-structure-y}) above, we obtain
$|m_1 - m_2| < m_1 + m_2$.
This clearly implies $m_1 \geq 0$ and $m_2 \geq 0$.
So $(m_1, m_2)$ is positive, as required.
\end{proof}

\begin{lem}\label{L_3Y14_Aut}
Let the hypotheses be as in Lemma~\ref{L_3Y21_Ord_Kth}
and Lemma~\ref{lem-order-structure-K_0},
let $\gamma \colon K_0 (A) \to \Q \oplus \Q$ be as in
Lemma~\ref{L_3Y21_Ord_Kth},
and let $\ld$ be as in Lemma~\ref{lem-order-structure-K_0}.
Assume that $\lambda > 0$.
Then there is exactly one nontrivial automorphism $\sm$ of $K_0 (D)$
as a scaled ordered group, given by
$[\gm \circ \sm \circ \gm^{-1}] (x, y) = (x, -y)$ for $x, y \in \Q$.
\end{lem}

\begin{proof}
By Lemma~\ref{lem-order-structure-K_0},
it suffices to prove that $\Q \oplus \Q$,
with the scaled ordered group structure in that lemma,
has the unique nontrivial automorphism $(x, y) \mapsto (x, -y)$.

It is immediate to check that
the map  $(x, y) \mapsto (x, -y)$ is indeed an automorphism.
Denote it by $\sigma$ for the purpose of this proof.
It is well known, and easy to check,
that any group homomorphism between rational vector spaces is $\Q$-linear.
Therefore, any automorphism of $\Q \oplus \Q$ is $\Q$-linear,
and in particular, is given by an element of $GL_2 (\Q)$.
We write elements of $G$ from now on as column vectors
rather than row vectors.
Suppose
$Q = \left( \begin{smallmatrix} a & b \\ c  & d \end{smallmatrix} \right)$
implements an automorphism of  $\Q \oplus \Q$.
Because
$Q \left( \begin{smallmatrix} 1 \\ 0 \end{smallmatrix} \right)
  = \left( \begin{smallmatrix} 1 \\ 0 \end{smallmatrix} \right)$,
we have $a = 1$ and $c = 0$.
Then $d \neq 0$; otherwise the matrix is not invertible.
Multiplying by $\sigma$ on the left reverses the sign of $d$,
so we may therefore assume without loss of generality that $d > 0$.
Replacing $Q$ with $Q^{-1}$ if needed,
we may furthermore assume without loss of generality that $d \geq 1$.
Conjugating by $\sigma$ reverses the sign of~$b$ but does not change~$d$.
so we may moreover assume without loss of generality that $b \geq 0$.

We now claim that $b = 0$ and $d=1$.
Choose some rational $y \in (0, \frac{1}{\lambda})$.
Then $(1, -y)$ is positive.
For any $n \in \N$ we have
\[
Q^n \left( \begin{matrix} 1 \\ -y \end{matrix} \right)
 = \left( \begin{matrix} 1 - b y \sum_{k = 0}^{n - 1} d^k
  \\ -d^n y \end{matrix} \right)  \, .
\]
If $b  > 0$, then for large enough $n$
we have $1 - b y \sum_{k = 0}^{n - 1} d^k < 0$.
Thus $Q^n \left( \begin{smallmatrix} 1 \\ -y \end{smallmatrix} \right)$
is not positive,
so $Q$ does not preserve the positive cone.
Therefore $b=0$.
This means that
$Q^n \left( \begin{smallmatrix} 1 \\ -y \end{smallmatrix} \right)
  = \left( \begin{smallmatrix} 1  \\ -d^n y \end{smallmatrix} \right)$.
If $d > 1$, for sufficiently large $n$ we have $|-d^n y| > \ld^{-1}$,
again showing that the positive cone is not preserved.
Thus, $d = 1$, as claimed.
\end{proof}

We now describe the trace space and pairing with $K$-theory
in a special case of Construction~\ref{D_3X28_GVS_df},
in which we assume a high degree of symmetry
(although less than earlier in this section)
and that $X_1, X_2, \ldots, X_m$ are contractible.
In the next theorem and its proof, we will write elements of the join
$\Dt_1 \star \Dt_2 \star \cdots \star \Dt_m$
of simplexes $\Dt_1, \Dt_2, \ldots, \Dt_m$ as (formal)
convex combinations $\sm = \sum_{j = 1}^m \lambda_j \sm_j$
with $\sm_j \in \Dt_j$ and $\ld_j \in [0, 1]$ for $j = 1, 2, \ldots, m$
satisfying $\sum_{j = 1}^m \lambda_j = 1$.

\begin{thm}\label{thm_trace_space}
Adopt the notation of Construction~\ref{D_3X28_GVS_df},
with the additional conditions in Lemma~\ref{L_3Y26_Simple}.
Assume that the connecting maps
are chosen as in Lemma~\ref{L_3Y26_Simple},
so that the direct limit $A$ is simple.
Further assume that, for all $n \in \N$,
whenever $(k,j)$ and $(k',j')$ are pairs of distinct elements
of $\{1,2,\ldots m\}$, we have $\mu_{k, j} (n) =  \mu_{k', j'}(n)$,
and that for any $j,j' \in \{1,2,\ldots,m\}$,
we have $\mu_{j, j} (n) =  \mu_{j', j'}(n)$.
Assume furthermore that $r_1 (0) = r_2 (0) = \cdots = r_m (0)$.
Finally assume that $\kp_j > 0$ for $j = 1, 2, \ldots, m$.
Let $\bigl( (D_n)_{n \in \Nz}, (\ps_{n_2, n_1})_{n_2 \geq n_1} \bigr)$
be the AF system defined in Lemma~\ref{L_3Y21_AFK},
let $D$ be its direct limit, and let $\io \colon D \to A$ be as there.
Let $\Omega$ be the Poulsen simplex,
and let $\theta \colon \Omega \to \{\pt\}$
be the function which collapses $\Omega$ to a point.
Then there are isomorphisms
$R \colon \T (D) \to \{\pt\}^{\star m}$ (the standard $m$-simplex)
and $S \colon \T (A) \to \Omega^{\star m}$
(the join of $m$ copies of the Poulsen simplex)
such that:
\begin{enumerate}
\item\label{I_3Z11_TDpt}
$\T (\io) \colon \T (A) \to \T (D)$
is given by $\T (\io) = R^{-1} \circ \theta^{\star m} \circ S$.
\item\label{I_3Z11_Pairing}
The pairing between $K$-theory and traces is given by
$\tau_* ([p]) =  \T (\io) (\tau)_* (\io_*^{-1}([p])$
for any \pj{} $p \in M_{\I} (A)$.
\item\label{I_3Z11_Ind}
If $\sm_1, \sm_2, \ldots, \sm_k \in \Omega$,
and $\sm = \sum_{j = 1}^m \lambda_j \sm_j$
is a formal convex combination in $\Omega^{\star m}$,
then $S^{-1} (\sm)_* \colon K_0 (A) \to \R$
depends only on $\lambda_1, \lambda_2, \ldots, \lambda_m$,
and not on the choice of $\sm_1, \sm_2, \ldots, \sm_m$.
\end{enumerate}
\end{thm}

\begin{proof}
If $R \colon \Dt \to \Sm$ is \ct{} and affine,
then let $\Aff (R) \colon \Aff (\Sm) \to \Aff (\Dt)$
be the function $\Aff (R) (f) = f \circ R$.

The hypotheses and
Construction \ref{D_3X28_GVS_df}(\ref{I_3X30_GVS_Cumulative}) imply that
\begin{equation}\label{Eq_24627_17St}
r_1 (n) = r_2 (n) = \cdots = r_m (n)
\end{equation}
for all $n \in \Nz$.
Denote the common value by $r (n)$.

We modify the connecting maps $\ph_{n + 1, \, n} \colon A_n \to A_{n + 1}$
and define $\widetilde{\ph}_{n + 1, \, n} \colon A_n \to A_{n + 1}$
as follows.
For $n \in \Nz$, set $t_n = r (n + 1) - d_j (n + 1) \in \Nz$,
which is not zero because of the conditions in Lemma~\ref{L_3Y26_Simple},
and is divisible by $r (n)$ by~(\ref{Eq_24627_17St}).
By the case $m = 1$ of Lemma~\ref{L_3Y26_Simple},
there are diagonal unital \hm{s}
\[
\widetilde{\et}_n^j
\colon C \bigl( X_{j}^{s_j (n)}, \, M_{r (n)} \bigr)
\to C \bigl( X_{j}^{s_j (n + 1)}, \, M_{t_n} \bigr)
\]
such that for $n \in \Nz$ and $j = 1, 2, \ldots, m$, with
\[
\widetilde{\ph}_{n + 1, \, n}^{j, j} (f)
  = \diag \bigl( f \circ P_{n, j, 1}, \, f \circ P_{n, j, 2}, \, \ldots,
  \, f \circ P_{n, j, d_j (n + 1)}, \, \widetilde{\et}_n^j (f)  \bigr) \, ,
\]
the direct limit ${\widetilde{A}}_j$ of the system
$\bigl( C ( X_{j}^{s_j (n)} )_{n \in \Nz},
    (\widetilde{\ph}^{j, j}_{n_2, n_1})_{n_2 \geq n_1} \bigr)$
is simple.
For $j \neq k$ and $n_2 \geq n_1$,
set $\widetilde{\ph}^{j, k}_{n_2, n_1} = 0$.

Thus, the direct limit ${\widetilde{A}}$ of
$\bigl( (A_n)_{n \in \Nz},
 (\widetilde{\ph}_{n_2 , n_1} )_{n_2 \geq n_1} \bigr)$
is the direct sum
${\widetilde{A}} = \bigoplus_{j = 1}^{m} {\widetilde{A}}_j$.
The algebras ${\widetilde{A}}_j$ are one seed Villadsen algebras
in the sense of Construction~\ref{D_3X28_GVS_df},
and also they are Villadsen algebras in the sense of \cite[Section~2]{ELN}.
Let ${\widetilde{\kp}}_j$ be the number
of Construction \ref{D_3X28_GVS_df}(\ref{I_3Y22_GVS_Kp})
for ${\widetilde{A}}_j$.
Then one easily checks that ${\widetilde{\kp}}_j = \kp_j$
(as in Construction \ref{D_3X28_GVS_df}(\ref{I_3Y22_GVS_Kp}) for~$A$),
so the algebras ${\widetilde{A}}_j$
all have strictly positive radius of comparison,
by the case $m = 1$ of Proposition~\ref{P_3X28_GVS_rc}.
By \cite[Theorem~4.5]{ELN}, there are affine \hme{s}
${\widetilde{S}}_j \colon \T ( {\widetilde{A}}_j ) \to \Om$.
By \cite[Lemma 2.10]{HP_asymmetry}, there is an \im{}
$\Aff ( \T ( {\widetilde{A}} ) )
 \to \bigoplus_{j = 1}^{m} \Aff ( T ( {\widetilde{A}}_j ) )$,
with the maximum norm on the direct sum.
This isomorphism is easily seen to have the form $\Aff (J)$
for the map
$J \colon \T ( {\widetilde{A}}_1 ) \star \T ( {\widetilde{A}}_2 )
   \star \cdots \star \T ( {\widetilde{A}}_m ) \to \T ( {\widetilde{A}} )$
which sends the formal convex combination $\sum_{j = 1}^m \lambda_j \ta_j$,
with $\ta_j \in \T ( {\widetilde{A}}_j)$ for $j = 1, 2, \ldots, m$,
to the \tst{} $\ta$ given by
$\ta (a_1, a_2, \ldots, a_m) = \sum_{j = 1}^m \lambda_j \ta_j (a_j)$
for $a_j \in {\widetilde{A}}_j$ for $j = 1, 2, \ldots, m$.
In particular, $J$ is an affine \hme.
Therefore
\[
{\widetilde{S}}
 = \bigl( {\widetilde{S}}_1 \star {\widetilde{S}}_2
      \star \cdots \star {\widetilde{S}}_m \bigr) \circ J^{- 1}
   \colon \T ( {\widetilde{A}} ) \to \Om^{ \star m}
\]
is an affine \hme.

Apply Lemma~\ref{L_3Y21_AFK} to the direct system
$\bigl( (A_n)_{n \in \Nz},
  \, (\widetilde{\ph}_{n_2, n_1})_{n_2 \geq n_1} \bigr)$.
We get an AF~system
$\bigl( (D_n)_{n \in \Nz},
  \, (\widetilde{\ps}_{n_2, n_1})_{n_2 \geq n_1} \bigr)$,
using the same algebras $D_n$ as before, but different maps.
Let $\widetilde{D}$ be the direct limit,
and let ${\widetilde{\io}} \colon {\widetilde{D}} \to {\widetilde{A}}$
be the map from Lemma~\ref{L_3Y21_AFK}.
Then $\widetilde{D}$ is a direct sum
${\widetilde{D}} = \bigoplus_{j = 1}^{m} {\widetilde{D}}_j$
of UHF algebras~${\widetilde{D}}_j$.
Moreover, for $j = 1, 2, \ldots, m$, there is a unital injective \hm{}
${\widetilde{\io}}_j \colon {\widetilde{D}}_j \to {\widetilde{A}}_j$,
as in Lemma~\ref{L_3Y21_AFK},
such that for
$d_1 \in {\widetilde{D}}_1, \, d_2 \in {\widetilde{D}}_2, \, \ldots, \,
   d_m \in {\widetilde{D}}_m$
we have
\[
{\widetilde{\io}} (d_1, d_2, \ldots, d_m)
 = \bigl( {\widetilde{\io}}_1 (d_1), \, {\widetilde{\io}}_2 (d_2), \,
   \ldots, {\widetilde{\io}}_m (d_m) \bigr).
\]
Since ${\widetilde{D}}_j$ has a unique \tst,
there is an obvious \im{}
${\widetilde{R}}_j \colon \T ( {\widetilde{D}}_j ) \to \{ \pt \}$.
The same reasoning as in the previous paragraph gives an affine \hme{}
$I \colon {\widetilde{D}}_1 \star {\widetilde{D}}_2
      \star \cdots \star {\widetilde{D}}_m \to \T ( {\widetilde{D}} )$
and that
\[
{\widetilde{R}}
 = \bigl( {\widetilde{R}}_1 \star {\widetilde{R}}_2
      \star \cdots \star {\widetilde{R}}_m \bigr) \circ I^{- 1}
   \colon \T ( {\widetilde{D}} ) \to \{ \pt \}^{ \star m}
\]
is also an affine \hme.
(Thus $\T (\widetilde{D})$ is the standard $m$-simplex.)
Since
${\widetilde{R}}_j \circ \T ( {\widetilde{\io}}_j )
  = \te \circ {\widetilde{S}}_j$,
we have
${\widetilde{R}} \circ \T ( {\widetilde{\io}} )
  = \te^{ \star m} \circ {\widetilde{S}}$.

For any unital \hm{} $\af \colon A \to B$ of \uca{s},
let $\af^{\wedge} \colon \Aff (\T (A)) \to \Aff (\T (B))$
be the map induced by $\T (\af) \colon \T (B) \to \T (A)$.
We have
\[
\bigl\| (\ph_{n + 1, n})^{\wedge}
     - (\widetilde{\ph}_{n + 1, n})^{\wedge} \bigr\|
 \leq \frac{2 [ r (n + 1) - \min_{1 \leq j \leq m} d_j (n + 1) ] }{r (n+1) }
   \, .
\]
Likewise,
\[
\bigl\| (\ps_{n + 1, n}) - (\widetilde{\ps}_{n + 1, n})^{\wedge} \bigr\|
  \leq \frac{2 [ r (n + 1) - \min_{1 \leq j \leq m} d_j (n + 1) ] }{r (n+1) }
   \, .
\]
Because $\min_{1 \leq j \leq m} \kappa_j > 0$, we necessarily have
\[
\sum_{n = 1}^{\infty}
  \frac{2 [ r (n + 1) - \min_{1 \leq j \leq m} d_j (n + 1) ] }{r (n + 1)}
\leq \sum_{j = 1}^m \sum_{n = 1}^{\infty}
  \frac{2 [ r (n + 1) - d_j (n + 1) ] }{r (n + 1)}
   < \infty \, .
\]
It now follows from \cite[Proposition 2.15]{HP_asymmetry}
that there exist isomorphisms
\[
\gm \colon \Aff( \T (\widetilde{D}) ) \to \Aff( \T (D) )
\andeqn
\rho \colon \Aff( \T (\widetilde{A}) ) \to \Aff( \T (A) )
\]
such that
$\io^{\wedge} \circ \gm = \rh \circ ({\widetilde{\io}})^{\wedge}$.
The maps $\rh$ and $\gm$ necessarily come from \im{s}
$G \colon \T (D) \to \T ( {\widetilde{D}} )$
and $H \colon \T (A) \to \T ( {\widetilde{A}} )$.
Thus we have the following commutative diagram:
\begin{equation}\label{Eq_24625_TrDg}
\begin{CD}
\T (A) @>{H}>> \T ( {\widetilde{A}} ) @>{\widetilde{S}}>>
           \Om^{ \star m}  \\
@VV{\T (\io)}V  @VV{\T ( {\widetilde{\io}} )}V
             @VV{\te^{ \star m}}V  \\
\T (D) @>{G}>> \T ( {\widetilde{D}} ) @>{\widetilde{R}}>>
           \{ \pt \}^{ \star m}.
\end{CD}
\end{equation}
Define $R = {\widetilde{R}} \circ G$ and $S = {\widetilde{S}} \circ H$.
We now have the maps $R$ and~$S$,
and we have proved~(\ref{I_3Z11_TDpt}).

Part~(\ref{I_3Z11_Pairing}) of the conclusion
is now proved by unravelling the definitions, as follows.
Since $\io_* \colon K_0 (D) \to K_0 (A)$ is an \im{} of ordered groups,
we may assume $[p] = \io_* ([q])$ for some \pj{} $q \in \Mi (D)$.
Then
\[
\T (\io) (\ta)_{*} ( (\io_*)^{-1} ([p]))
  = (\ta \circ \io)_{*} ( [q] )
  = \ta_* ([p]).
\]

We prove part~(\ref{I_3Z11_Ind}).
For $j = 1, 2, \ldots, m$ let $\ld_j \in [0, 1]$,
and suppose $\sum_{j = 1}^m \lambda_j = 1$.
Further, for $k = 1, 2$ let $\sm_j^{(k)} \in \Om$,
and set $\sm^{(k)} = \sum_{j = 1}^m \lambda_j \sm_j^{(k)}$.
Set $\ta^{(k)} = S^{-1} (\sm^{(k)})$.
We must prove that the maps
$(\ta^{(1)})_*, \, (\ta^{(2)})_* \colon K_0 (A) \to \R$
are equal.
It is enough to consider their values on classes in $K_0 (A)$ of \pj{s}.
So let $p \in \Mi (A)$ be a \pj.
Let $f \in \Aff (\T (A))$ be the function $f (\ta) = \ta (p)$
for $\ta \in \T (A)$.
We need to prove that $f (\ta^{(1)}) = f (\ta^{(2)})$.

Choose a \pj{} $q \in \Mi (D)$ such that $[p] = \io_* ([q])$,
and let $g \in \Aff (\T (D))$ be $g (\sm) = \sm (q)$ for $g \in \T (D)$.
Apply the contravariant functor
$\Dt \mapsto \Aff (\Dt)$ to~(\ref{Eq_24625_TrDg}),
use the definitions of $R$ and~$S$, and omit the middle column, getting
\[
\begin{CD}
\Aff (  \{ \pt \}^{ \star m} ) @>{\Aff ( R )}>> \Aff ( \T (D) )  \\
@VV{\Aff (  \te^{ \star m} )}V  @VV{\io^{\wedge}}V            \\
\Aff (  \Om^{ \star m} ) @>{\Aff ( S )}>> \Aff ( \T (A) ).
\end{CD}
\]
We have $\io^{\wedge} (g) = f$ using part~(\ref{I_3Z11_Pairing}).
Therefore, for $k = 1, 2$,
\[
\begin{split}
f (\ta^{(k)})
& = \bigl[ \Aff (S) \circ \Aff (\te^{ \star m}) \circ \Aff (R)^{-1} \bigr]
    (g) (\ta^{(k)})
\\
& = \bigl( g \circ R^{-1} \circ \te^{ \star m} \circ S \bigr) (\ta^{(k)})
  = \bigl( g \circ R^{-1} \circ \te^{ \star m} \bigr) (\sm^{(k)}) \, .
\end{split}
\]
Obviously $\te^{ \star m} (\sm^{(1)}) = \te^{ \star m} (\sm^{(2)})$,
so $f (\ta^{(1)}) = f (\ta^{(2)})$, as desired.
\end{proof}

\begin{cor}\label{C_24626_New26}
Let the hypotheses be as in Lemma~\ref{lem-order-structure-K_0},
and assume that the number $\ld$ of Lemma~\ref{lem-order-structure-K_0}
satisfies $\ld > 0$.
Let $\sm$ be the nontrivial automorphism of $K_0 (A)$
of Lemma~\ref{L_3Y14_Aut}.
Then $\sm$ extends to an automorphism of the Elliott invariant of~$A$.
\end{cor}

\begin{proof}
The hypotheses imply $K_1 (A) = 0$.
Therefore we need only find an automorphism $F$ of $\T (A)$
such that, for all $\et \in K_0 (A)$ and $\ta \in \T (A)$, we have
\begin{equation}\label{Eq_24626_Cmpt}
F (\ta)_* (\sm (\et)) = \ta_* (\et) \, .
\end{equation}

Let the notation be as in Theorem~\ref{thm_trace_space},
including the corresponding AF~algebra~$D$,
the map $\io \colon D \to A$, $R \colon \T (D) \to \{\pt\} \star \{\pt\}$,
$S \colon \T (A) \to \Om \star \Om$, etc.
Let $F_0 \colon \Om \star \Om \to \Om \star \Om$
and $G_0 \colon \{\pt\} \star \{\pt\} \to \{\pt\} \star \{\pt\}$
be the automorphisms which exchange the two factors.
Define $F = S^{-1} \circ F_0 \circ S$.

Since $D$ is AF,
there is $\bt \in \Aut (D)$
such that $\bt_* = \io_*^{- 1} \circ \sm \circ \io_*$.
In the notation of Lemma~\ref{lem-order-structure-K_0},
the formulas $\om_1 (x, y) = x + \ld y$ and $\om_2 (x, y) = x - \ld y$
define distinct states on $\Q \oplus \Q$.
Moreover, if ${\widetilde{\sm}}$ is the automorphism of $\Q \oplus \Q$
corresponding to~$\sm$, then $\om_2 = \om_1 \circ {\widetilde{\sm}}$.
Therefore $\bt$ induces a nontrivial automorphism $G$ of $\T (D)$.
Since $G_0$ is the unique nontrivial automorphism of $\{\pt\} \star \{\pt\}$,
it follows that $G = R^{-1} \circ G_0 \circ R$.
Since $G_0 \circ ( \te \star \te) = ( \te \star \te) \circ F_0$,
Theorem \ref{thm_trace_space}(\ref{I_3Z11_TDpt}) implies
\begin{equation}\label{Eq_24626_GTTF}
G \circ \T (\io) = \T (\io) \circ F \, .
\end{equation}
By definition, $G (\ta) = \ta \circ \bt$ for all $\ta \in \T (D)$.
So, using $\bt_*^2 = \id_{K_0 (D)}$, for $\mu \in K_0 (D)$ we get
\begin{equation}\label{Eq_24626_GtaSt}
G (\ta)_* ( \bt_* (\mu))
 = G (\ta)_* ( \bt_*^{-1} (\mu))
 = (\ta \circ \bt)_* (\bt_*^{-1} (\mu))
 = \ta_* (\mu) \, .
\end{equation}
Therefore, using (\ref{Eq_24626_GtaSt}) at the third step
and (\ref{Eq_24626_GTTF}) at the fourth step,
\[
\begin{split}
F (\ta)_{*} ( (\sm \circ \io_*) (\mu))
& = F (\ta)_{*} ( (\io \circ \bt)_* (\mu))
 = \T (\io) (F (\ta))_{*} ( \bt_* (\mu))
\\
& = G (\T (\io) (\ta))_{*} ( \bt_* (\mu))
  = \T (\io) (\ta)_* (\mu)
  = \ta_* ( \io_* (\mu) ) \, .
\end{split}
\]
Since $\io_*$ is surjective, (\ref{Eq_24626_Cmpt})~follows.
\end{proof}

\section{Local radius of comparison and construction of nonisomorphic
 C*-algebras}\label{S_lrc}

We now combine the technical results from the previous sections
to construct an uncountable family of pairwise nonisomorphic AH algebras
with the same Elliott invariant and same radius of comparison.
Our distinguishing invariant is what we call
the \emph{local radius of comparison}.
We define it initially on $V (A)$,
the semigroup of Murray--von Neumann equivalence classes
of projections in $K \otimes A$,
but under mild assumptions it is well defined as a function on $K_0 (A)_{+}$.

\begin{dfn}\label{D_3Y14_lrc}
Let $A$ be a C*-algebra.
We define $\lrc_A \colon V (A) \to [0, \infty]$ by
$\lrc ([p]) = \rc ( p (K \otimes A) p )$.
We write $\lrc$ when $A$ is understood.
\end{dfn}

If $A$ has quotients which are not stably finite, one should instead
use $r_{ p (K \otimes A) p }$, as after \cite[Definition 3.2.2]{BRTTW}.

\begin{lem}\label{L_3Y14_lrc_wd}
The function $\lrc_A$ of Definition~\ref{D_3Y14_lrc} is well defined.
\end{lem}

\begin{proof}
If $p$ and $q$ are Murray--von Neumann equivalent projections
in $K \otimes A$,
then $p (K \otimes A) p  \cong q (K \otimes A) q$,
and in particular those C*-algebras have the same radius of comparison.
\end{proof}

\begin{cor}\label{C_3Y14_lrc_K0}
Let $A$ be a unital \ca.
\begin{enumerate}
\item\label{I_3Y14_lrc_wd_cnc}
If $A$ has cancellation of projections, then $\lrc_A$ descends to
a well defined function on $K_0 (A)_{+}$.
\item\label{I_Y14_lrc_wd_sr1}
If $A$ has stable rank one, then $\lrc_A$ descends to
a well defined function on $K_0 (A)_{+}$.
\end{enumerate}
\end{cor}

\begin{proof}
For part~(\ref{I_3Y14_lrc_wd_cnc}), cancellation of projections
implies that the standard map $V (A) \to K_0 (A)$ is injective.
For part~(\ref{I_Y14_lrc_wd_sr1}), stable rank one
implies cancellation of projections.
\end{proof}

\begin{rmk}\label{R_3Y14_lrc_soft}
It is natural to further extend the definition of $\lrc$ to the positive cone,
and from there to the soft part of the Cuntz semigroup as well.
For the purposes of this paper,
we have no use for such an extension,
as we do not know how to compute the Cuntz semigroup of our examples.
Therefore, for the sake of simplicity,
we restrict our domain to $K_0 (A)_{+}$.
(The more general situation is considered in~\cite{Asd}.)
It could be interesting to see if one can construct examples
which have the same Elliott invariant,
same radius of comparison,
same local of radius of comparison function defined on $K_0 (A)_{+}$,
but can be distinguished by considering the local radius of comparison
extended to the soft part of the Cuntz semigroup.
See Problem~\ref{Pb_240620_CuNotV}.
\end{rmk}

We know very little about abstract algebraic properties
of the local radius of comparison function,
except for the following immediate observations, whose proofs we omit.
See~\cite{Asd} for more.

\begin{lem}\label{lem-lrc-immediate}
Let $A$ be a C*-algebra, and let $p,q$ be projections in $K \otimes A$.
\begin{enumerate}
\item\label{I_3Y14_ord}
If $[p] \leq [q]$ then $\lrc ([p]) \geq \lrc ([q])$.
\item\label{I_3Y14_homog}
$\lrc (n \cdot [p]) = \frac{1}{n}\lrc ([p])$.
\end{enumerate}
\end{lem}

\begin{rmk}\label{R_3Y14_1_over}
Lemma~\ref{lem-lrc-immediate} shows that
the function $\frac{1}{\lrc}$ is homogeneous and increasing,
which suggests that it, rather than $\lrc$,
might be a more natural object to study in further depth.
We do not further explore this direction here.
\end{rmk}

We now construct our examples.
We separate several lemmas for readability.
In the notation of Construction~\ref{D_3X28_GVS_df}
and Lemma~\ref{L_3Y21_Ord_Kth}, the system
$\bigl ( (A_n)_{n \in \Nz}, (\ph_{n_2, n_1})_{n_2 \geq n_1} \bigr)$
in Lemma~\ref{L_3Y23_rc_corner} below
is a diagonal two seed Villadsen system with the following parameters:
\begin{enumerate}
\item\label{I_3Y24_m}
$m = 2$, and $X_1$ and $X_2$ are solid with $\dim (X_1) = \dim (X_2) = h$.
\item\label{I_3Y24_dsj}
$d_1 (n) = l (n) - c (n)$, $d_2 (n) \leq l (n) - c (n)$,
$d_1 (1) = d_2 (1) = d$,
and $s_j (n) = \prod_{k = 1}^{n} d_j (k)$ for $j = 1, 2$ and $n \in \N$.
\item\label{I_3Y24_mujkn}
For $j, k \in \{ 1, 2 \}$ and $n \in \N$,
\[
\mu_{k, j} (n) = \begin{cases}
   l (n) - c (n) & \hspace*{1em} j = k
        \\
   c (n)         & \hspace*{1em} j \neq k.
\end{cases}
\]
\item\label{I_3Y24_rjn}
For $j = 1, 2$,
$r_j (0) = 1$ and $r_j (n) = \prod_{k = 1}^{n} l_j (k)$
for $n \in \N$.
\item\label{I_3Y24_kp}
$\kp_1 > \frac{1}{2}$.
\end{enumerate}

\begin{lem}\label{L_3Y23_corner_kp}
Assume parts (\ref{I_3X28_GVS_m}), (\ref{I_3Y23_GVS_Amp}),
(\ref{I_3X28_GVS_Sizes}), and (\ref{I_3X30_GVS_Cumulative})
of Construction~\ref{D_3X28_GVS_df},
and let the notation be as there,
except that we do not specify the sequence $(d_2 (n) )_{n \in \N}$,
and therefore also do not specify $s_2 (n)$.
Further assume the hypotheses of parts
(\ref{I_L_3Y21_Ord_Kth_2}), (\ref{I_L_3Y21_Ord_Kth_Eq}),
and~(\ref{I_L_4624_Ord_Kth_l2c1})
of Lemma~\ref{L_3Y21_Ord_Kth}.
Let $\kp_1$ be as in Construction \ref{D_3X28_GVS_df}(\ref{I_3Y22_GVS_Kp}).
Assume further:
\begin{enumerate}
\item\label{L_3Y23_rc_corner_d}
$d_1 (1) \geq 2$.
\item\label{L_3Y24_rc_corner_NZ}
$d_1 (n) = l (n) - c (n)$ for all $n \in \N$.
\item\label{L_3Z11_rc_corner_half}
$\kp_1 > \frac{1}{2}$.
\item\label{Eq_3Y24_rc_tDfn}
$(t (n) )_{n \in \Nz}$ is the sequence in $\Nz$
inductively defined by $t (0) = 0$ and
\[
t (n + 1) = d_1 (n + 1) t (n) + c (n + 1) [ r (n) - t (n) ]
\]
for $n \in \Nz$.
\setcounter{TmpEnumi}{\value{enumi}}
\end{enumerate}
Then:
\begin{enumerate}
\setcounter{enumi}{\value{TmpEnumi}}
\item\label{Eq_3Y23_rc_rPrd}
For $n \in \Nz$, the numbers $r_1 (n)$ and $r_2 (n)$ of
Construction \ref{D_3X28_GVS_df}(\ref{I_3X30_GVS_Cumulative})
are both given by
\[
r_1 (n) = r_2 (n) = \prod_{k = 1}^{n} l (k).
\]
\item\label{Eq_3Y23_rc_kpp}
The limit
\[
\kp_1' = \limi{n} \frac{s_1 (n)}{r (n) - t (n)}
\]
exists and is strictly positive.
\item\label{Eq_3Y24_rc_Int}
We have
\[
\kp_1' > \kp_1
\andeqn
\kp_1' - \kp_1 < \min \left( \kp_1, \, \frac{\kp_1^2}{\kp_1' - \kp_1} \right).
\]
\end{enumerate}
\end{lem}

\begin{proof}
Part~(\ref{Eq_3Y23_rc_rPrd}) is in Lemma~\ref{L_3Y21_Ord_Kth},
and is independent of $(d_2 (n) )_{n \in \N}$.
Call the common value $r (n)$.
Then, by Lemma~\ref{L_3Y23_kp_exist},
the number $\kp_1 = \limi{n} \frac{s_1 (n)}{r (n)}$ exists
and has the same value
regardless of the choice of $(d_2 (n) )_{n \in \Nz}$.

We claim that
\begin{equation}\label{Eq_3Y23_Decr}
0 = \frac{t (0)}{r (0)}
  < \frac{t (1)}{r (1)}
  < \frac{t (2)}{r (2)}
  < \cdots
  < 1 - \kp_1.
\end{equation}
The proof of the claim,
except with the upper bound $\frac{1}{2}$ in place of $1 - \kp_1$,
is the same as that of \cite[Lemma 1.3]{HP_asymmetry},
with $c (n)$ in place of $k (n)$ with $d_1 (n)$ in place of $d (n)$,
using $2 c (n) < l (n)$ for all $n \in \N$
(hypothesis (\ref{I_L_4624_Ord_Kth_l2c1}) of Lemma~\ref{L_3Y21_Ord_Kth}).
To prove the upper bound here, we first show by induction that
\begin{equation}\label{Eq_3Y28_tnrnsn}
\frac{t (n)}{r (n)} \leq 1 - \frac{s_1 (n)}{r (n)}
\end{equation}
for all $n \in \Nz$.
This is true for $n = 0$ since $t (0) = 0$.
If~(\ref{Eq_3Y28_tnrnsn}) holds for~$n$, then,
using the induction hypothesis on the first term at the second step,
\[
\begin{split}
\frac{t (n + 1)}{r (n + 1)}
& = \frac{d_1 (n + 1) t (n)}{[d_1 (n + 1) + c (n + 1)] r (n)}
    + \frac{c (n + 1) \cdot [r (n) - t (n)]}{[d_1 (n + 1) + c (n + 1)] r (n)}
\\
& \leq \frac{d_1 (n + 1) }{d_1 (n + 1) + c (n + 1)}
           \left( 1 - \frac{s_1 (n)}{r (n)} \right)
    + \frac{c (n + 1) }{d_1 (n + 1) + c (n + 1)}
\\
& = 1 - \frac{d_1 (n + 1) }{d_1 (n + 1) + c (n + 1)}
         \left( \frac{s_1 (n)}{r (n)} \right)
 = 1 - \frac{s_1 (n + 1)}{r (n + 1)},
\end{split}
\]
as desired.
This completes the induction.
Now the first part of~(\ref{Eq_3Y23_Decr}) implies that
for all $k > n$ we have
\[
\frac{t (n)}{r (n)}
 < \frac{t (k)}{r (k)}
 \leq 1 - \frac{s_1 (k)}{r (k)} \, ,
\]
so
\[
\frac{t (n)}{r (n)}
 < \frac{t (n + 1)}{r (n + 1)}
 \leq \limi{k} \left( 1 - \frac{s_1 (k)}{r (k)} \right)
 = 1 - \kp_1 \, .
\]
This proves the claim.

It follows that $\limi{n} \frac{t (n)}{r (n)}$ exists and satisfies
\[
1 - \kp_1
 \geq \limi{n} \frac{t (n)}{r (n)}
 \geq \frac{t (1)}{r (1)}
 = \frac{c (1)}{d_1 (1) + c (1)}.
\]
Since $\limi{n} \frac{s_1 (n)}{r (n)} = \kp_1$ and $\kp_1 > \frac{1}{2}$,
we deduce that the number $\ld = \limi{n} \frac{t (n)}{s_1 (n)}$
exists and satisfies
\begin{equation}\label{Eq_3Y23_ld}
0 < \frac{c (1)}{d_1 (1) + c (1)}
  \leq \kp_1 \ld
  < 1 - \kp_1
  < \frac{1}{2}.
\end{equation}
It follows that the limit $\kp_1'$ in~(\ref{Eq_3Y23_rc_kpp}) exists
and satisfies
\[
\kp_1' = \frac{1}{\kp_1^{- 1} - \ld} = \frac{\kp_1}{1 - \kp_1 \ld} > 0 \, .
\]
Using this, and applying $x \mapsto (1 - x)^{-1}$ to~(\ref{Eq_3Y23_ld}),
we get
\[
1 < \left[ 1 - \frac{c (1)}{d_1 (1) + c (1)} \right]^{-1}
  \leq \frac{\kp_1'}{\kp_1}
  < \frac{1}{\kp_1}
  < 2,
\]
so that $\kp_1 < \kp_1' < 2 \kp_1$.
This gives the first part of~(\ref{Eq_3Y24_rc_Int}).
It also implies $\kp_1' - \kp_1 < \kp_1 < 1$,
so that $(\kp_1' - \kp_1)^2 < \kp_1^2$, whence
\[
\kp_1' - \kp_1 < \frac{\kp_1^2}{\kp_1' - \kp_1}.
\]
This and $\kp_1' - \kp_1 < \kp_1$
give the second part of~(\ref{Eq_3Y24_rc_Int}).
\end{proof}

\begin{lem}\label{L_3Y23_rc_corner}
Adopt the hypotheses and notation of
Construction~\ref{D_3X28_GVS_df}, Lemma~\ref{L_3Y21_Ord_Kth},
and Lemma~\ref{L_3Y23_corner_kp},
except that we do not need $X_1$ and $X_2$ to be contractible.
We now do specify the sequences $(d_2 (n) )_{n \in \Nz}$
and $(s_2 (n) )_{n \in \Nz}$.
Assume that there is $d \in \N$
with $d \geq 2$ such that $d_1 (1) = d_2 (1) = d$.
Let $\kp_1$ and $\kp_2$
be as in Construction \ref{D_3X28_GVS_df}(\ref{I_3Y22_GVS_Kp}).
Let $h \in \N$, and assume that
$X_1$ and $X_2$ are solid with $\dim (X_1) = \dim (X_2) = h$.
Let $p_1, p_2 \in A$ be the \pj{s}
$p_1 = \ph_{\I, 0} (1, 0)$ and $p_2 = \ph_{\I, 0} (0, 1)$.
If
\[
\kp_1' - \kp_1 \leq \kp_2 \leq \frac{\kp_1^2}{\kp_1' - \kp_1},
\]
then
\[
\rc (p_1 A p_1) = \frac{\kp_2 \kp_1' h}{2 (\kp_1' - \kp_1)}
\andeqn
\rc (p_2 A p_2) = \frac{\kp_1 \kp_1' h}{2 (\kp_1' - \kp_1)}.
\]
\end{lem}

\begin{proof}
For simplicity of notation,
we abbreviate the images $\ph_{n, 0} (1, 0)$ and $\ph_{n, 0} (0, 1)$
in $A_n$ to $p_1$ and $p_2$, regardless of~$n$.
Also, by Lemma \ref{L_3Y23_corner_kp}(\ref{Eq_3Y23_rc_rPrd}),
we have
\begin{equation}\label{Eq_24621_Two}
r_1 (n) = r_2 (n)
\end{equation}
for all $n \in \Nz$.
Call the common value $r (n)$.

We claim that $p_1 A p_1$ is the direct limit of a
diagonal two seed Villadsen system with the parameters
described in (\ref{I_3Y24_p1_m})--(\ref{I_3Y24_p1_kp}) below,
following the notation of
Construction~\ref{D_3X28_GVS_df} but writing $X_j^{(1)}$, $d_j^{(1)}$,
etc.\  in place of $X_j$, $d_j$, etc.
We don't take the bottom level to be at $n = 0$,
because the summand in $p_1 A p_1$ using $X_2$ is absent.
Therefore we take the bottom level to be
$A_1^{(1)} = p_1 A_1 p_1 = C (X_1^d, M_{d}) \oplus C (X_2^d, M_{c (1)})$.
We preserve the indexing here, to get $A_n^{(1)} \S A_n$,
so that the indexing is shifted by~$1$ from Construction~\ref{D_3X28_GVS_df}:
the bottom level is at $n = 1$ instead of $n = 0$.
Thus,
\begin{equation}\label{Eq_24621_One}
s_j^{(1)} (1) = 1,
\quad
s_j^{(1)} (n) = \prod_{k = 2}^n d_j^{(1)} (k),
\quad
r_1^{(1)} (1) = r (1) - c (1),
\quad
r_2^{(1)} (1) = c (1),
\end{equation}
etc.
Then:
\begin{enumerate}
\item\label{I_3Y24_p1_m}
$m^{(1)} = 2$, $X_1^{(1)} = X_1^d$, and $X_2^{(1)} = X_2^d$.
\item\label{I_3Y24_p1_dsj}
$d_j^{(1)} (n) = d_j (n)$
and $s_j^{(1)} (n) = d^{- 1} s_j (n)$ for $j = 1, 2$ and $n \geq 2$.
\item\label{I_3Y24_p1_mujkn}
For $j, k \in \{ 1, 2 \}$ and $n \in \N$,
\[
\mu_{k, j}^{(1)} (n) = \begin{cases}
   l (n) - c (n) & \hspace*{1em} j = k
        \\
   c (n)         & \hspace*{1em} j \neq k.
\end{cases}
\]
\item\label{I_3Y24_p1_rjn}
For $n \in \N$,
$r_1^{(1)} (n) = r (n) - t (n)$ and $r_2^{(1)} (n) = t (n)$.
\item\label{I_3Y24_p1_kp}
We have
\[
\kp_1^{(1)} = \frac{\kp_1'}{d}
\andeqn
\kp_2^{(1)} = \frac{\kp_2 \kp_1'}{d (\kp_1' - \kp_1)}.
\]
\setcounter{TmpEnumi}{\value{enumi}}
\end{enumerate}

We prove the claim.
Item~(\ref{I_3Y24_p1_m}) is immediate.
It is easy to see that the passage from $A_n$ to $p_1 A_n p_1$
does not change the fact that the maps are diagonal,
the multiplicity matrices of the maps,
or the maps of spaces underlying the partial maps of the algebras.
Therefore the new system is still diagonal,
and (\ref{I_3Y24_p1_dsj}) (except for $s_j^{(1)} (n)$)
and~(\ref{I_3Y24_p1_mujkn}) hold.
The equation $s_j^{(1)} (n) = d^{- 1} s_j (n)$
results from the change of starting level and $d_j (1) = d$;
see (\ref{Eq_24621_One}).

Item~(\ref{I_3Y24_p1_rjn}) follows by induction on~$n$,
by using item~(\ref{I_3Y24_p1_mujkn}) above, the assumption in
Lemma \ref{L_3Y23_corner_kp}(\ref{L_3Y24_rc_corner_NZ}),
and the formula for $r (n)$ in Lemma~\ref{L_3Y21_Ord_Kth}
to compare the recursion formula
of Lemma~\ref{L_3Y23_corner_kp}(\ref{Eq_3Y24_rc_tDfn})
with Construction \ref{D_3X28_GVS_df}(\ref{I_3X30_GVS_Cumulative}),
starting with
$t (1) = c (1)$, $r (1) - t (1) = d_1 (1)$, and~(\ref{Eq_24621_One}).
For~(\ref{I_3Y24_p1_kp}), we now get,
using the definition of $\kp_1'$
in Lemma \ref{L_3Y23_corner_kp}(\ref{Eq_3Y23_rc_kpp}) at the last step,
\[
\kp_1^{(1)}
 = \limi{n} \frac{s_1^{(1)} (n)}{r_1^{(1)} (n)}
 = \limi{n} \frac{s_1 (n)}{d [r (n) - t (n)]}
 = \frac{\kp_1'}{d}
\]
and, using Construction \ref{D_3X28_GVS_df}(\ref{I_3Y22_GVS_Kp})
and the definition of $\kp_1'$ at the fourth step,
\[
\begin{split}
\kp_2^{(1)}
& = \limi{n} \frac{s_2^{(1)} (n)}{r_2^{(1)} (n)}
  = \left( \limi{n} \frac{s_2^{(1)} (n)}{s_1^{(1)} (n)} \right)
      \left( \limi{n} \frac{s_1^{(1)} (n)}{t (n)} \right)
\\
& =
      \left( \limi{n} \frac{s_2 (n)}{s_1 (n)} \right) \cdot \frac{1}{d} \cdot
      \left( \limi{n} \frac{s_1 (n)}{t (n)} \right)
  = \left( \frac{\kp_2}{\kp_1} \right)
    \cdot \frac{1}{d} \cdot \left( \frac{1}{ \frac{1}{\kp_1} -
    \frac{1}{\kp_1'}} \right)
  = \frac{\kp_2 \kp_1'}{d (\kp_1' - \kp_1)},
\end{split}
\]
as desired.
The claim is proved.

Applying Proposition~\ref{P_3X28_GVS_rc}
and using $\dim (X_1^{(1)}) = \dim (X_1^{(2)}) = d h$,
we get
\[
\rc (p_1 A p_1)
 = \max \left( \frac{\kp_1' h}{2},
   \, \frac{\kp_2 \kp_1' h}{2 (\kp_1' - \kp_1)} \right).
\]
Since $\kp_1' - \kp_1 \leq \kp_2$ we get
$\kp_2 \kp_1' (\kp_1' - \kp_1)^{- 1} \geq \kp_1'$,
so the claimed value of $\rc (p_1 A p_1)$ follows.

Using analogous notation,
we claim that $p_2 A p_2$ is the direct limit of a
diagonal two seed Villadsen system with the following parameters,
with the only differences being in~(\ref{I_3Y24_ptwo_rjn}),
where the values are exchanged from~(\ref{I_3Y24_p1_rjn}),
and in~(\ref{I_3Y24_ptwo_kp}):
\begin{enumerate}
\setcounter{enumi}{\value{TmpEnumi}}
\item\label{I_3Y24_ptwo_m}
$m^{(2)} = m^{(1)}$, $X_1^{(2)} = X_1^{(1)}$, and $X_2^{(2)} = X_2^{(1)}$.
\item\label{I_3Y24_ptwo_dsj}
$d_j^{(2)} (n) = d_j^{(1)} (n)$
and $s_j^{(2)} (n) = s_j^{(1)}(n)$ for $j = 1, 2$ and $n \geq 2$.
\item\label{I_3Y24_ptwo_mujkn}
For $j, k \in \{ 1, 2 \}$ and $n \in \N$,
$\mu_{j, k}^{(2)} (n) = \mu_{j, k}^{(1)} (n)$.
\item\label{I_3Y24_ptwo_rjn}
For $n \in \N$,
$r_1^{(2)} (n) = t (n)$ and $r_2^{(2)} (n) = r (n) - t (n)$.
\item\label{I_3Y24_ptwo_kp}
For $j = 1, 2$,
\[
\kp_1^{(2)} = \frac{\kp_1 \kp_1'}{d (\kp_1' - \kp_1)}
\andeqn
\kp_2^{(2)} = \frac{\kp_2 \kp_1'}{d \kp_1}.
\]
\end{enumerate}

Items (\ref{I_3Y24_ptwo_m}), (\ref{I_3Y24_ptwo_dsj}),
and~(\ref{I_3Y24_ptwo_mujkn}) follow for the same reasons as before.
Item~(\ref{I_3Y24_ptwo_rjn}) follows from item~(\ref{I_3Y24_p1_rjn})
in the previous claim (for $p_1 A p_1$),
because $p_1 + p_2 = 1$ and the matrix sizes are $r (n)$.
For the first part of~(\ref{I_3Y24_ptwo_kp}),
using (\ref{I_3Y24_ptwo_dsj}), (\ref{I_3Y24_ptwo_rjn}),
and~(\ref{I_3Y24_p1_dsj}) at the second step,
and using the definition of $\kp_1$
from Construction \ref{D_3X28_GVS_df}(\ref{I_3Y22_GVS_Kp})
and the definition of $\kp_1'$
from Lemma \ref{L_3Y23_corner_kp}(\ref{Eq_3Y23_rc_kpp}) at the fourth step,
we have
\[
\begin{split}
\kp_1^{(2)}
& = \limi{n} \frac{s_1^{(2)} (n)}{r_1^{(2)} (n)}
  = \limi{n} \frac{s_1 (n)}{d t (n)}
\\
& = \frac{1}{d} \limi{n}
  \frac{1}{\frac{r (n)}{s_1 (n)} - \frac{r (n) - t (n)}{s_1 (n)}}
  = \frac{1}{d} \left( \frac{1}{ \frac{1}{\kp_1} - \frac{1}{\kp_1'}} \right)
  = \frac{\kp_1 \kp_1'}{d (\kp_1' - \kp_1)} \, .
\end{split}
\]
For the second part of~(\ref{I_3Y24_ptwo_kp}),
using item~(\ref{I_3Y24_ptwo_rjn}) at the second step,
using items (\ref{I_3Y24_ptwo_dsj})
and~(\ref{I_3Y24_p1_dsj}) at the third step,
and using (\ref{Eq_24621_Two}) and the definitions
of $\kp_1$, $\kp_2$, and $\kp_1'$ at the fourth step, we have
\[
\begin{split}
\kp_2^{(2)}
& = \limi{n} \frac{s_2^{(2)} (n)}{r_2^{(2)} (n)}
  = \left( \limi{n} \frac{s_2^{(2)} (n)}{s_1^{(2)} (n)} \right)
      \left( \limi{n} \frac{s_1^{(2)} (n)}{r (n) - t (n)} \right)
\\
& = \left( \limi{n} \frac{s_2 (n)}{s_1 (n)} \right)
      \left( \limi{n} \frac{s_1 (n)}{d [r (n) - t (n)]} \right)
  = \frac{\kp_2 \kp_1'}{d \kp_1}.
\end{split}
\]
The claim is proved.

Applying Proposition~\ref{P_3X28_GVS_rc} as before,
we get
\[
\rc (p_2 A p_2)
 = \max \left( \frac{\kp_1 \kp_1' h}{2 (\kp_1' - \kp_1)},
    \, \frac{\kp_2 \kp_1' h}{2 \kp_1} \right).
\]
Now $\kp_2 \leq \kp_1^2 (\kp_1' - \kp_1)^{- 1}$ implies
\[
\frac{\kp_2 \kp_1'}{\kp_1} \leq \frac{\kp_1 \kp_1'}{\kp_1' - \kp_1},
\]
so the claimed value of $\rc (p_2 A p_2)$ follows.
\end{proof}

\begin{lem}\label{L_3Y26_Choose_dn}
Let $\kp \in (0, 1)$ and let $l_0 \in \N$ satisfy $1 - \frac{1}{l_0} > \kp$.
Then there is a sequence $(l (n))_{n \in \N}$ in $\N$
such that:
\begin{enumerate}
\item\label{I_3Y26_Choose_dn_Pr}
$\prod_{n = 1}^{\I} \left( 1 - \frac{1}{l (n)} \right) = \kp$.
\item\label{I_3Y26_Choose_dn_Inc}
$(l (n))_{n \in \N}$ is nondecreasing.
\item\label{I_3Y26_Choose_dn_l0}
$l (n) \geq l_0$ for all $n \in \N$.
\item\label{I_3Y26_Choose_dn_is_l0}
$l (1) = l_0$.
\item\label{I_3Y26_Choose_dn_lim}
$\limi{n} l (n) = \I$.
\item\label{I_3Y26_Choose_dn_div}
For every $k \in \N$ there exist $n_1, n_2 \in \N$
such that $k$ divides both $l (n_1)$ and $l (n_2) - 2$.
\setcounter{TmpEnumi}{\value{enumi}}
\end{enumerate}
\end{lem}

\begin{proof}
Clearly $l_0 \geq 2$.
We construct the sequence $(l (n))_{n \in \N}$ in stages.
Throughout, given $l (1), l (2), \ldots, l (n)$, we let
\begin{equation}\label{Eq_24621_Five}
\gm_n = \prod_{k = 1}^{n} \left( 1 - \frac{1}{l (k)} \right).
\end{equation}
(We suppress the dependence on $l (1), l (2), \ldots, l (n)$.)

Choose $l_1 \in \N$ so large that
\begin{equation}\label{Eq_24621_Three}
l_1 > \max \left( l_0, \, \frac{1}{\kp} + 2 \right)
\andeqn
\left( 1 - \frac{1}{l_0} \right) \left( 1 - \frac{1}{l_1} \right) > \kp.
\end{equation}
Let $n_0$ be the largest integer such that
\[
\left( 1 - \frac{1}{l_0} \right) \left( 1 - \frac{1}{l_1} \right)^{n_0 - 2}
 > \kp.
\]
Then $n_0 \geq 3$.
We set $l (1) = l_0$ and $l (n) = l_1$ for $n = 2, 3, \ldots, n_0 - 1$.
So
\begin{equation}\label{Eq_24621_Four}
\gm_{n_0 - 1}
 = \left( 1 - \frac{1}{l_0} \right)
       \left( 1 - \frac{1}{l_1} \right)^{n_0 - 2}
\andeqn
1 - \frac{1}{l_1} \leq \frac{\kp}{\gm_{n_0 - 1}}.
\end{equation}

Let $l_2$ be the least integer such that
\[
\frac{1}{l_2} < 1 - \frac{\kp}{\gm_{n_0 - 1}}.
\]
Then, using (\ref{Eq_24621_Four}) and the first part
of~(\ref{Eq_24621_Three}) for the first part,
\begin{equation}\label{Eq_3Y27_ch_l1lll2}
l_2 > l_1 > \frac{1}{\kp} + 1
\andeqn
\frac{1}{l_2} < 1 - \frac{\kp}{\gm_{n_0 - 1}} \leq \frac{1}{l_2 - 1}.
\end{equation}
The second part of~(\ref{Eq_3Y27_ch_l1lll2}) implies
\begin{equation}\label{Eq_24621_Six}
\left( 1 - \frac{1}{l_2 - 1} \right) \gm_{n_0 - 1}
 \leq \kp
 < \left( 1 - \frac{1}{l_2} \right) \gm_{n_0 - 1}.
\end{equation}
Therefore, using the first part of~(\ref{Eq_3Y27_ch_l1lll2})
at the last step,
\begin{equation}\label{Eq_3Z11_gmkpl2}
\begin{split}
\left( 1 - \frac{1}{l_2} \right) \gm_{n_0 - 1} - \kp
& \leq \left[ \left( 1 - \frac{1}{l_2} \right)
      - \left( 1 - \frac{1}{l_2 - 1} \right) \right] \gm_{n_0 - 1}
\\
& = \frac{\gm_{n_0 - 1}}{l_2 (l_2 - 1)}
  < \frac{1}{l_2 (l_2 - 1)}
  < \frac{\kp}{l_2}
  < \frac{1}{l_2}.
\end{split}
\end{equation}
Define $l (n_0) = l_2$.

We construct the rest of the sequence $(l (n))_{n \in \N}$
by induction on~$n$.
Define sequences $(\af_n)_{n \geq n_0 + 1}$
and $(\bt_n)_{n \geq n_0 + 1}$ by, for $n \geq 1$,
\begin{equation}\label{Eq_3Y27_afbt}
\af_{n + n_0} = \begin{cases}
   0 & \hspace*{1em} {\mbox{$n$ is odd}}
        \\
   2 & \hspace*{1em} {\mbox{$n$ is even}}
\end{cases}
\andeqn
\bt_{n + n_0}  = \begin{cases}
   \frac{n + 1}{2} & \hspace*{1em} {\mbox{$n$ is odd}}
        \\
   \frac{n}{2}     & \hspace*{1em} {\mbox{$n$ is even}}.
\end{cases}
\end{equation}
For $n \geq n_0 + 1$, we then have
\begin{equation}\label{Eq_3Y27_btnn}
\bt_n \leq \frac{1}{2} ( 1 + n - n_0).
\end{equation}

Recalling the definition of $\gm_n$ (see~(\ref{Eq_24621_Five})),
the sequence we construct will then be required to satisfy
the following for all $n \geq n_0 + 1$:
\begin{enumerate}
\setcounter{enumi}{\value{TmpEnumi}}
\item\label{I_3Y27_Ind_gg}
$l (n) > l (n - 1)$.
\item\label{I_I_3Y27_Ind_lkp}
$l (n) > \kp^{- 1} + n - n_0 + 2$.
\item\label{I_I_3Y27_Ind_div}
$\bt_n | [ l (n) - \af_n ]$.
\item\label{I_I_3Y27_Ind_ord}
$\kp < \gm_n < \gm_{n - 1} < 1$.
\item\label{I_I_3Y27_Ind_est}
$\gm_n - \kp < \frac{1}{l (n)}$.
\end{enumerate}
In the induction step, we will only use (\ref{I_I_3Y27_Ind_lkp}),
(\ref{I_I_3Y27_Ind_ord}), and (\ref{I_I_3Y27_Ind_est}).
We therefore verify these when $n = n_0$.
Using the first parts
of (\ref{Eq_3Y27_ch_l1lll2}) and~(\ref{Eq_24621_Three}),
condition~(\ref{I_I_3Y27_Ind_lkp}) is
$l (n_0) = l_2 > l_1 > \frac{1}{\kp} + 2$,
in condition~(\ref{I_I_3Y27_Ind_ord}) the first inequality
uses~(\ref{Eq_24621_Six}) and the rest is immediate,
and condition~(\ref{I_I_3Y27_Ind_est}) is~(\ref{Eq_3Z11_gmkpl2})
and $l (n_0) = l_2$.

Given $n \geq n_0$ and $l (1), l (2), \ldots, l (n)$
such that (\ref{I_I_3Y27_Ind_lkp}), (\ref{I_I_3Y27_Ind_ord}),
and~(\ref{I_I_3Y27_Ind_est}) hold,
let $m_0$ be the least integer such that
\begin{equation}\label{Eq_3Y27_ch_m0}
\frac{1}{m_0} < 1 - \frac{\kp}{\gm_n}.
\end{equation}
Let $m$ be the least integer such that $m \geq m_0$ and
\begin{equation}\label{Eq_24621_Nine}
\bt_{n + 1} | ( m - \af_{n + 1} ).
\end{equation}
Then
\begin{equation}\label{Eq_24621_Seven}
m - m_0 \leq \bt_{n + 1} - 1.
\end{equation}
Also, using~(\ref{I_I_3Y27_Ind_est}) at the fourth step,
(\ref{I_I_3Y27_Ind_lkp}) at the fifth step,
and (\ref{Eq_3Y27_btnn}) at the last step,
\[
\frac{1}{m}
  \leq \frac{1}{m_0}
  < 1 - \frac{\kp}{\gm_n}
  = \frac{1}{\gm_n} (\gm_n - \kp)
  < \frac{1}{l (n)}
  < \frac{1}{n - n_0 + 2}
  \leq \frac{1}{2 \bt_{n + 1}}.
\]
Since $m$, $l (n)$, and $2 \bt_{n + 1}$ are integers, we get,
using (\ref{I_I_3Y27_Ind_lkp}) at the end of the
second calculation,
\begin{equation}\label{Eq_3Y27_Cnsq}
m \geq 2 \bt_{n + 1} + 1
\andeqn
m \geq l (n) + 1 > \kp^{- 1} + n - n_0 + 3.
\end{equation}
In particular, also using~(\ref{Eq_24621_Seven}),
\begin{equation}\label{Eq_3Y27_m0m}
0 < m - \bt_{n + 1} \leq m_0 - 1.
\end{equation}
Since $m_0$ is the least integer satisfying~(\ref{Eq_3Y27_ch_m0}), we have
\[
\frac{1}{m}
 \leq \frac{1}{m_0}
 < 1 - \frac{\kp}{\gm_{n}}
 \leq \frac{1}{m_0 - 1}
 \leq \frac{1}{m - \bt_{n + 1}}.
\]
Therefore
\begin{equation}\label{Eq_24621_Eight}
\left( 1 - \frac{1}{m - \bt_{n + 1}} \right) \gm_{n}
 \leq \kp
 < \left( 1 - \frac{1}{m} \right) \gm_{n}.
\end{equation}
So, using the first part of~(\ref{Eq_3Y27_Cnsq}) at the third step,
\begin{equation}\label{Eq_3Y27_gmkpl2}
\begin{split}
\left( 1 - \frac{1}{m} \right) \gm_{n} - \kp
& \leq\left[ \left( 1 - \frac{1}{m} \right)
      - \left( 1 - \frac{1}{m - \bt_{n + 1}} \right) \right] \gm_{n}
\\
& = \left( \frac{\gm_n}{m} \right)
    \left( \frac{\bt_{n + 1}}{m - \bt_{n + 1}} \right)
  < \frac{\gm_n}{m}
  < \frac{1}{m}.
\end{split}
\end{equation}
Now define $l (n + 1) = m$.
Items (\ref{I_3Y27_Ind_gg}) and~(\ref{I_I_3Y27_Ind_lkp})
of the induction hypothesis for $n + 1$ are in~(\ref{Eq_3Y27_Cnsq}).
Item~(\ref{I_I_3Y27_Ind_div}) is~(\ref{Eq_24621_Nine})
and item~(\ref{I_I_3Y27_Ind_ord}) is (\ref{Eq_24621_Five})
and the second inequality in~(\ref{Eq_24621_Eight}).
Item~(\ref{I_I_3Y27_Ind_est}) follows from
(\ref{Eq_24621_Five}) and~(\ref{Eq_3Y27_gmkpl2}).
This completes the induction.

Parts (\ref{I_3Y26_Choose_dn_Inc}), (\ref{I_3Y26_Choose_dn_l0}),
and~(\ref{I_3Y26_Choose_dn_is_l0}) of the conclusion are now immediate.
Part~(\ref{I_3Y26_Choose_dn_div})
follows from (\ref{I_I_3Y27_Ind_div}) and~(\ref{Eq_3Y27_afbt}).
Part~(\ref{I_3Y26_Choose_dn_lim})
follows from (\ref{I_3Y27_Ind_gg}).
We have $\gm_n > \kp$ for all $n \in \N$ by~(\ref{I_I_3Y27_Ind_ord}),
and part~(\ref{I_3Y26_Choose_dn_Pr}) of the conclusion now
follows from (\ref{I_I_3Y27_Ind_est}) and $\limi{n} l (n) = \I$.
\end{proof}

\begin{lem}\label{L_3Y26_ChSeq}
Let $(l (n) )_{n \in \N}$ be a sequence in $\N$
such that the number
\[
\kp_0 = \prod_{n = 1}^{\I} \left( 1 - \frac{1}{l (n)} \right)
\]
satisfies $\kp_0 > 0$.
Then for every $\kp \in [0, \kp_0]$,
there is a sequence $(d (n) )_{n \in \N}$ in $\N$
such that $1 \leq d (n) \leq l (n) - 1$ for all $n \in \N$
and
\[
\prod_{n = 1}^{\I} \frac{d (n)}{l (n)} = \kp.
\]
\end{lem}

\begin{proof}
Since $\kp_0 > 0$, we have $l (n) \geq 2$ for all $n \in \N$
and $\limi{n} l (n) = \I$.
If $\kp = 0$, we can take $d (n) = 1$ for all $n \in \N$.
Therefore we assume $\kp > 0$.

For $n \in \Nz$, set
\[
\rh_n = \prod_{k = n + 1}^{\I} \left( 1 - \frac{1}{l (k)} \right).
\]
Then $\rh_n \leq 1$ for all $n \in \Nz$
and, since $\kp_0 > 0$, we also have $\limi{n} \rh_n = 1$.

We construct $d (n)$ by induction on~$n$ so that
$1 \leq d (n) \leq l (n) - 1$ and, with
\[
\gm_n = \prod_{k = 1}^{n} \frac{d (k)}{l (k)},
\]
we have
\[
\gm_n \rh_n \geq \kp
\andeqn
\gm_n \rh_n - \kp < \frac{1}{l (n)}.
\]
Given this,
since $\limi{n} l (n) = \I$ and $\limi{n} \rh_n = 1$,
we have $\limi{n} \gm_n = \kp$.
It now is clear that the resulting sequence $(d (n) )_{n \in \N}$
satisfies the conclusion of the lemma.

For the base case, let $m$ be the least integer such that
\begin{equation}\label{Eq_3Y26_base}
\frac{m \rh_1}{l (1)} \geq \kp.
\end{equation}
Since $\kp > 0$ and $\rh_1 > 0$, we have $m \geq 1$.
Also, $m \leq l (1) - 1$, because
 \[
 \frac{ [  l (1) - 1 ] \rh_1}{l (1)} = \rh_0 = \kp_0 \geq \kp .
 \]
Thus, if we take $d (1) = m$ then
\[
1 \leq d (1) \leq l (1) - 1
\andeqn
\frac{(m - 1) \rh_1}{l (1)} < \kp.
\]
Therefore, using $\rh_1 \leq 1$,
\[
0 \leq \frac{m \rh_1}{l (1)} - \kp
  < \frac{\rh_1}{l (1)}
  \leq \frac{1}{l (1)}.
\]
Since $\gm_1 = m / l (1)$, the base case is done.

Now suppose that $d (1), d (2), \ldots, d (n)$ are given,
$\gm_n$ is as above, and $\gm_n \rh_n \geq \kp$.
Let $m$ be the least integer such that
\begin{equation}\label{Eq_3Y26_Ind}
\frac{\gm_n m \rh_{n + 1}}{l (n + 1)} \geq \kp.
\end{equation}
Since $\kp > 0$, we have $\gm_n > 0$ and $\rh_{n + 1} > 0$, so $m \geq 1$.
Also $m \leq l (n + 1) - 1$, because
\[
\frac{\gm_n [  l (n + 1) - 1 ] \rh_{n + 1}}{l (n + 1)}
 = \gm_n \rh_{n} \geq \kp \, .
\]
Thus, if we take $d (n + 1) = m$ then
\[
1 \leq d (n+1) \leq l (n + 1) - 1
\andeqn
\frac{\gm_n (m - 1) \rh_{n + 1}}{l (n + 1)} < \kp.
\]
Therefore, using $\gm_n, \rh_n \leq 1$,
\[
0 \leq \frac{\gm_n m \rh_{n + 1}}{l (n + 1)} - \kp
  < \frac{\gm_n \rh_{n + 1}}{l (n + 1)}
  \leq \frac{1}{l (n + 1)}.
\]
Since $\gm_{n + 1} \rh_{n + 1} = \gm_n m \rh_{n + 1} / l (n + 1)$,
the induction step is done.
\end{proof}

\begin{thm}\label{T_3Y14_main}
Let $\omega \in (0, \infty)$.
There exists an uncountable family
of pairwise nonisomorphic simple unital AH~algebras
with the same Elliott invariant and radius of comparison~$\omega$,
all of which have stable rank one.
\end{thm}

Before giving the proof,
we describe the algebras and how we distinguish them.
For each~$\om$ there will be a nonempty open interval $I \S \R$,
a real number $c > 0$, a number $\rh \not\in I$,
a scaled ordered group $(G, G_{+}, g_0)$,
and an unordered pair of elements $\et_1, \et_2 \in G_{+}$.
For each $\bt \in I$
we will construct a two seed Villadsen algebra~$A_{\bt}$
for which both seed spaces are the same cube,
and whose Elliott invariants consist of
\[
\bigl( K_0 (A_{\bt}), \, K_0 (A_{\bt})_{+}, \, [1_{A_{\bt}}] \bigr)
 \cong (G, G_{+}, g_0),
\quad
K_1 (A_{\bt}) = 0,
\quad {\mbox{and}} \quad
\T (A_{\bt}) \cong \Om \star \Om,
\]
the join of two copies of the Poulsen simplex,
and with $\rc (A_{\bt}) = \om$.
Moreover, the pairing of $K_0 (A_{\bt})$ and $\T (A_{\bt})$
is always the same, regardless of~$\bt$.
For any isomorphism
$\te \colon G \to K_0 (A_{\bt})$ of scaled ordered groups,
using the local radius of comparison function as in
Corollary \ref{C_3Y14_lrc_K0}(\ref{I_Y14_lrc_wd_sr1}), we will have
\[
\lrc_{A_{\bt}} (\te (\et_1)) = c \bt
\andeqn
\lrc_{A_{\bt}} (\te (\et_2)) = c \rh
\]
or
\[
\lrc_{A_{\bt}} (\te (\et_1)) = c \rh
\andeqn
\lrc_{A_{\bt}} (\te (\et_2)) = c \bt.
\]
Given the above, these relations determine $\bt$ uniquely.
Thus, our family will be parametrized via the local radius of comparison
by an open interval in~$\R$.

The proof, below, has the following corollary.

\begin{cor}\label{C_24626_New}
For every algebra $B$ as constructed
in the proof of Theorem~\ref{T_3Y14_main},
there is an automorphism of the Elliott invariant $\Ell (B)$
which is not induced by any automorphism of~$B$.
\end{cor}

\begin{proof}
We refer to the proof of Theorem~\ref{T_3Y14_main}.
The algebra constructed are all of the form $M_N (A_{\beta})$,
for algebras $A_{\beta}$ with stable rank one and
satisfying the hypotheses of Corollary~\ref{C_24626_New26}.
By Corollary~\ref{C_24626_New26}, there is an
automorphism of $\Ell (A_{\beta})$, so also of $\Ell (M_N (A_{\beta}))$,
for which the constituent automorphism
$\sm \colon K_0 (A_{\beta}) \to K_0 (A_{\beta})$ satisfies,
following (\ref{Eq_24626_NewSt}) below,
$\sm ( [ p_1^{(\bt)} ] ) = [ p_2^{(\bt)} ]$.
Since $\bt < \kp_1$, we have
$\rc \bigl( p_1^{(\bt)} A_{\bt} p_1^{(\bt)} \bigr)
 \neq \rc \bigl( p_2^{(\bt)} A_{\bt} p_2^{(\bt)} \bigr)$,
so by Corollary \ref{C_3Y14_lrc_K0}(\ref{I_Y14_lrc_wd_sr1})
there is no \am{} $\af \colon M_N (A_{\beta}) \to M_N (A_{\beta})$
such that $\af_* = \sm$.
\end{proof}

\begin{proof}[Proof of Theorem~\ref{T_3Y14_main}]
We first assume $\om > \frac{1}{2}$.

Let $h \in \N$ be the least integer such that $h > 2 \omega$.
Define $\kappa_1 = \frac{2 \omega}{h}$.
Then $\frac{1}{2} < \kappa_1 < 1$.
Choose $l_0 \in \N$ such that $1 - \frac{1}{l_0} > \kp_1$ and $l_0 \geq 3$.
Apply Lemma~\ref{L_3Y26_Choose_dn} with $\kp = \kp_1$
and this value of~$l_0$.
Let $(l (n))_{n \in \N}$ be the resulting sequence.
Define $c (n) = 1$ and $d_1 (n) = l (n) - 1$ for all $n \in \N$.
The construction in Lemma~\ref{L_3Y26_Choose_dn}
ensures that $l (n) \geq 3$ for all $n \in \N$.
Therefore $l (n) > 2 c (n)$ for all $n \in \N$,
and we may use these choices in Lemma~\ref{L_3Y23_corner_kp}.
For $n \in \Nz$, define $r_1 (n) = r_2 (n) = \prod_{k = 1}^n l (k)$
and, following Construction \ref{D_3X28_GVS_df}(\ref{I_3Y23_GVS_Amp}),
define $s_1 (n) = \prod_{k = 1}^n d_1 (k)$.
Then, by Lemma \ref{L_3Y23_corner_kp}(\ref{Eq_3Y23_rc_rPrd})
and the choices made using Lemma~\ref{L_3Y26_Choose_dn},
we have $\kp_1 = \limi{n} \frac{s_1 (n)}{r_1 (n)}$,
that is, $\kp_1$ is the same as the number $\kp_1$
in Construction \ref{D_3X28_GVS_df}(\ref{I_3Y22_GVS_Kp}).
Let $\kp_1'$ be as in Lemma \ref{L_3Y23_corner_kp}(\ref{Eq_3Y23_rc_kpp}).
Then, by Lemma \ref{L_3Y23_corner_kp}(\ref{Eq_3Y24_rc_Int}),
\[
\kp_1' - \kp_1
 < \min \left( \kp_1, \, \frac{\kp_1^2}{\kp_1' - \kp_1} \right).
\]

For any $\bt$ with
\begin{equation}\label{Eq_3Y28_Use_rc_Int}
\kp_1' - \kp_1
 < \bt
 < \min \left( \kp_1, \, \frac{\kp_1^2}{\kp_1' - \kp_1} \right),
\end{equation}
we construct a diagonal two seed Villadsen algebra $A_{\bt}$.
This will be done in such a way that $\rc (A_{\bt}) = \om$,
that the algebras $A_{\bt}$ for $\bt$ as in~(\ref{Eq_3Y28_Use_rc_Int})
all have the same Elliott invariant,
and that if $\bt \neq \gm$ then $A_{\bt} \not \cong A_{\gm}$.

Let $\bt$ satisfy~(\ref{Eq_3Y28_Use_rc_Int}).
Apply Lemma~\ref{L_3Y26_ChSeq} with $(l (n) )_{n \in \N}$ as given,
with $\kp_1$ in place of $\kp_0$, and with $\bt$ in place of $\kp$.
Let $( d_2^{(\bt)} (n) )_{n \in \N}$ be the resulting sequence,
so that
\begin{equation}\label{Eq_24626_Two}
\prod_{n = 1}^{\I} \frac{d_2^{(\bt)} (n)}{l (n)} = \bt \, .
\end{equation}
Set $X_1 = X_2 = [0, 1]^{h}$.
Following the hypothesis in
Lemma \ref{L_3Y21_Ord_Kth}(\ref{I_L_3Y21_Ord_Kth_Eq}),
take the matrix $\mu (n)$
of Construction \ref{D_3X28_GVS_df}(\ref{I_3X28_GVS_Sizes}) to be
\[
\mu (n)
= \left( \begin{matrix} d_1 (n) & 1 \\ 1  & d_1 (n) \end{matrix} \right)
= \left(
   \begin{matrix} l (n) - 1 & 1 \\ 1   & l (n) - 1 \end{matrix} \right).
\]
By Lemma~\ref{L_3Y26_Simple}, there is a two seed Villadsen system
$\bigl( (A_n^{(\bt)})_{n \in \Nz},
 (\ph_{n_2, n_1}^{(\bt)})_{n_2 \geq n_1} \bigr)$
with diagonal maps, and with, following Lemma~\ref{L_3Y21_Ord_Kth}
and Construction~\ref{D_3X28_GVS_df},
the sequences $(l (n) )_{n \in \N}$, $(c (n) )_{n \in \N}$,
$(d_1 (n) )_{n \in \N}$, $( d_2^{(\bt)} (n) )_{n \in \Nz}$,
and $( \mu (n) )_{n \in \Nz}$,
and the spaces $X_1$ and $X_2$, as chosen,
such that $A_{\bt} = \dirlim_n A_n^{(\bt)}$ is simple.
Let $\ph_{\I, n}^{(\bt)} \colon A_n^{(\bt)} \to A_{\bt}$
be the map associated to the direct limit.

Clearly $K_1 (A_{\bt}) = 0$.
Set
\[
\lambda = \prod_{n = 1}^{\infty} \left( 1 - \frac{2 c (n)}{l (n)} \right) .
\]
By Lemma~\ref{lem-order-structure-K_0}, as a scaled ordered group,
$K_0 (A_{\bt}) \cong \Q \oplus \Q$, ordered as there,
using this value of~$\ld$,
and with $[1]$ corresponding to $(1, 0)$.
By Lemma~\ref{L_3Y21_AFK}, there are a simple AF algebra~$D$
(independent of~$\bt$),
and unital \hm{s} $\io_{\bt} \colon D \to A_{\bt}$
for $\bt$ satisfying~(\ref{Eq_3Y28_Use_rc_Int})
such that $(\io_{\bt})_{*}$ is an order isomorphism.
Let $\Om$ be the Poulsen simplex.
By Theorem~\ref{thm_trace_space}, $\T (D)$ has exactly two extreme points,
say $\sm_1$ and $\sm_2$,
and there is an isomorphism $S \colon \T (A_{\bt}) \to \Om \star \Om$
such that the pairing between $\T (A_{\bt})$ and $K_0 (A_{\bt})$
is given as follows.
If $\ta_1, \ta_2 \in \Om$
and $\ta = \ld_1 \ta_1 + \ld_2 \ta_2 \in \Om \star \Om$
is a formal convex combination,
then, for $\nu \in K_0 (D)$, we have
\[
S^{-1} (\ta)_* ( (\io_{\bt})_* (\nu))
 = \ld_1 (\sm_1)_* (\nu) + \ld_2 (\sm_2)_* (\nu).
\]
Since $(\io_{\bt})_*$ is an isomorphism of scaled ordered groups,
and the right hand side does not depend on~$\bt$,
this shows that the Elliott invariants of the algebras $A_{\bt}$
are all isomorphic.

Recall from parts (\ref{I_3Y23_GVS_Amp}) and~(\ref{I_3Y22_GVS_Kp})
of Construction~\ref{D_3X28_GVS_df}
that the number $\kp_2^{(\bt)}$ for $A_{\bt}$ is given by
\begin{equation}\label{Eq_24627_St30}
\kp_2^{(\bt)}
 = \limi{n} \frac{1}{r_2 (n)} \prod_{k = 1}^n d_2^{(\bt)} (k).
\end{equation}
Since $r_2 (n) = r_1 (n)$ and $d_2^{(\bt)} (n) \leq d_1 (n)$
for all $n \in \N$, we have $\kp_2^{(\bt)} \leq \kp_1$.
Since $X_1 = X_2$ and both have dimension~$h$,
Proposition~\ref{P_3X28_GVS_rc} implies that
\[
\rc ( A_{\bt} )
 = \frac{\kp_1 \dim (X_1)}{2}
 = \frac{\kp_1 h}{2}
 = \om.
\]
Thus, the algebras $A_{\bt}$ all have the same radius of comparison.

It follows from \cite[Theorem 4.1]{EHT}
that $A_{\bt}$ has stable rank one.

For $\bt$ satisfying~(\ref{Eq_3Y28_Use_rc_Int}), and with the
notation above, let
\begin{equation}\label{Eq_24626_NewSt}
p_1^{(\bt)} = \ph_{\I, 0}^{(\bt)} (1, 0) \in A_{\bt}
\andeqn
p_2^{(\bt)} = \ph_{\I, 0}^{(\bt)} (0, 1) \in A_{\bt}.
\end{equation}
We have $\kp_2^{(\bt)} = \bt$ by (\ref{Eq_24626_Two})
and~(\ref{Eq_24627_St30}).
Lemma~\ref{L_3Y23_rc_corner} and~(\ref{Eq_3Y28_Use_rc_Int})
therefore imply that
\begin{equation}\label{Eq_3Y28_rccAgm}
\rc \bigl( p_1^{(\bt)} A_{\bt} p_1^{(\bt)} \bigr)
  = \frac{\bt \kp_1' h}{2 (\kp_1' - \kp_1)}
\andeqn
\rc \bigl( p_2^{(\bt)} A_{\bt} p_2^{(\bt)} \bigr)
  = \frac{\kp_1 \kp_1' h}{2 (\kp_1' - \kp_1)}.
\end{equation}

Suppose now that $\bt$ and $\gm$ satisfy~(\ref{Eq_3Y28_Use_rc_Int}),
and that $\et \colon A_{\bt} \to A_{\gm}$ is an isomorphism.
We claim that $\bt = \gm$.
To prove the claim, first observe that
Lemma~\ref{L_3Y14_Aut} implies that
\[
\et_* \bigl( \bigl[ p_1^{(\bt)} \bigr] \bigr)
 = \bigl[ p_1^{(\gm)} \bigr]
\andeqn
\et_* \bigl( \bigl[ p_2^{(\bt)} \bigr] \bigr)
 = \bigl[ p_2^{(\gm)} \bigr]
\]
or
\[
\et_* \bigl( \bigl[ p_1^{(\bt)} \bigr] \bigr)
 = \bigl[ p_2^{(\gm)} \bigr]
\andeqn
\et_* \bigl( \bigl[ p_2^{(\bt)} \bigr] \bigr)
 = \bigl[ p_1^{(\gm)} \bigr].
\]
In the first case, stable rank one,
Corollary \ref{C_3Y14_lrc_K0}(\ref{I_Y14_lrc_wd_sr1}),
and~(\ref{Eq_3Y28_rccAgm}) imply that
\[
\frac{\bt \kp_1' h}{2 (\kp_1' - \kp_1)}
 = \frac{\gm \kp_1' h}{2 (\kp_1' - \kp_1)}.
\]
Therefore $\bt = \gm$.
In the second case, stable rank one,
Corollary \ref{C_3Y14_lrc_K0}(\ref{I_Y14_lrc_wd_sr1}),
and~(\ref{Eq_3Y28_rccAgm}) imply that
\[
\frac{\bt \kp_1' h}{2 (\kp_1' - \kp_1)}
 = \frac{\kp_1 \kp_1' h}{2 (\kp_1' - \kp_1)}.
\]
Therefore $\bt = \kp_1$, which contradicts~(\ref{Eq_3Y28_Use_rc_Int}).
The claim is proved, and the proof for $\om > \frac{1}{2}$ is complete.

The case $\om \leq \frac{1}{2}$ follows from the first case as follows.
Fix $N$ such that $N \omega > \frac{1}{2}$.
Find an uncountable family of pairwise nonisomorphic C*-algebras $A_{\beta}$
as above,
with $N \omega$ in place of $\omega$ and with $\beta$ in a suitable interval.
Then $M_N (A_{\beta})$ has radius of comparison $\omega$ for every~$\bt$,
and the Elliott invariant also doesn't depend on~$\bt$.
The C*-algebras $M_N(A_{\beta})$ are seen to be pairwise nonisomorphic
in precisely the same way as before:
the $K_0$ group of $M_N(A_{\beta})$ is the same as that of $A_{\beta}$,
except that the location of the unit
is replaced by $(N,0) \in \Q \oplus \Q$.
The local radius of comparison function,
as defined on the positive cone, is not changed, as it not affected by the
location of the unit. Because any automorphism of $\Q \oplus \Q$ as an abelian
group is $\Q$-linear, any automorphism which fixes $(N,0)$ must also fix
$(1,0)$. Thus, using Lemma~\ref{L_3Y14_Aut}, the only nontrivial automorphism
of $\Q \oplus \Q$ which preserves the positive cone and fixes $(N,0)$ is $(x,y)
\mapsto (x,-y)$. Therefore, the argument showing that the C*-algebras
$M_N(A_{\beta})$ are
pairwise nonisomorphic
is identical to the one for the C*-algebras $ A_{\beta}
$.
\end{proof}

\section{Open problems}\label{S_240620_Open}

We collect here some open problems and related comments.

The methods used here don't distinguish between algebras with
infinite radius of comparison:
if $A$ is simple, unital, and stably finite,
and $\rc (A) = \I$,
then $\rc (p A p) = 0$ for every \nzp{} $p \in A$,
by \cite[Theorem~2.18]{AGP}.

\begin{pbm}\label{Pb_240620_Inf_rc}
Is there an uncountable family
of nonisomorphic simple separable unital AH~algebras
with the same Elliott invariant and
infinite radius of comparison?
\end{pbm}

The algebras in the proof of Theorem~\ref{T_3Y14_main}
are not naturally classified
by the Elliott invariant and one extra number.
They are all direct limits of two seed Villadsen systems.
Presumably they have analogs
which are direct limits of $m$~seed Villadsen systems,
and for which, for each $\om \in (0, \I)$,
there is an $(m - 1)$-parameter family of nonisomorphic such algebras
with radius of comparison~$\omega$ and the same Elliott invariant,
which can be distinguished
by the radii of comparison of suitable corners.
In the spirit of~\cite{ELN}, we ask the following.

\begin{pbm}\label{Pb_240620_Cl_m_seed}
For $j = 1, 2$, let $m_j \in \N$
and let $A_j$ be a simple $m_j$-seed Villadsen algebra
using \fd{} solid seed spaces and diagonal maps.
Suppose there is
an isomorphism of the Elliott invariants of $A_1$ and~$A_2$
whose constituent isomorphism $f_0 \colon K_0 (A_1) \to K_0 (A_2)$
satisfies $\rc (p_1 (K \otimes A_1) p_1) = \rc (p_2 (K \otimes A_2) p_2)$
whenever $f_0 ([p_1]) = [p_2]$.
Does it follow that $A_1 \cong A_2$?
\end{pbm}

The hypothesis asks that there be an isomorphism of the pairs
consisting of the Elliott invariant and the functions $\lrc_{A_j}$
as in Definition~\ref{D_3Y14_lrc}.

In most cases, if $A$ is an $m$-seed Villadsen algebra,
we expect $\T (A)$ to be the join of $m$~copies of the Poulsen simplex
(see Theorem~\ref{thm_trace_space}),
so that at least the number of seeds can be recovered from
the Elliott invariant.

While it is at least plausible that Problem~\ref{Pb_240620_Cl_m_seed}
has a positive solution,
we do not expect it to generalize very far.
For example, it is plausible that there are examples as in the following
problems.

\begin{pbm}\label{Pb_240620_CuNotV}
Are there two simple separable unital AH~algebras $A_1$ and~$A_2$
with the same radius of comparison,
for which there is
an isomorphism of the Elliott invariants of $A_1$ and~$A_2$
whose constituent isomorphism $f_0 \colon K_0 (A_1) \to K_0 (A_2)$
satisfies $\rc (p_1 (K \otimes A_1) p_1) = \rc (p_2 (K \otimes A_2) p_2)$
whenever $f_0 ([p_1]) = [p_2]$,
but for which nonisomorphism can be proved using a nonunital version
of the radius of comparison for nonunital \hsa{s} of $A_1$ and~$A_2$?
Can $A_1$ and~$A_2$ be chosen to be direct limits of direct systems
with diagonal maps?
\end{pbm}

For $a \in A_{+}$, the radius of comparison of ${\ov{a A a}}$
should presumably be taken to be~$r_{A, a}$,
as in \cite[Definition 3.3.2]{BRTTW}.

\begin{pbm}\label{Pb_240620_NotELN}
Is there a simple separable unital AH~algebra~$A$ with diagonal maps
with the same Elliott invariant and radius of comparison
as one of the algebras classified in~\cite{ELN},
but not isomorphic to that algebra?
\end{pbm}

One can construct infinite seed Villadsen algebras from a direct
system in which the algebra at level~$n$ has $m_n$ summands,
with $m_n \to \I$.
Such a construction was used in~\cite{Hrs_exotic}.
Presumably there are versions in which the summands of $A_n$ are
parametrized by the vertices at level~$n$ of a rooted tree
in which all branches have infinite length.
Other constructions giving different algebras may well be possible.

\begin{pbm}\label{Pb_240620_BorelCx}
For fixed $\om \in (0, \I]$ and a fixed Elliott invariant,
what is the complexity, in the sense used in~\cite{Sbk},
of the isomorphism relation on all stably finite simple separable
nuclear unital \ca{s} $A$ with $\rc (A) = \om$, with the given
Elliott invariant, and with stable rank one?
What if one restricts to AH~algebras?
\end{pbm}

Since the \im{} relation for separable simple AI~algebras
is as complex as the \im{} relation for all separable \ca{s},
it is plausible that the complexity here is the same.
However, the proof in~\cite{Sbk}
depends on the complexity the \im{} relation for Choquet simplexes,
so the method used there does not apply directly.

\begin{pbm}\label{Pb_240625_Bauer}
Is there a simple stably finite unital \ca~$A$ with stable rank one
and strictly positive radius of comparison
such that $\T (A)$ is a Bauer simplex?
\end{pbm}

The example in~\cite{Vld2} has a unique \tst,
but doesn't have stable rank one.

There is no simple AH algebra~$A$ with diagonal maps
and strictly positive radius of comparison
such that $\T (A)$ is a Bauer simplex,
by \cite[Theorem 4.7 and Proposition~4.8]{EllNiu3}.

\begin{pbm}\label{Pb_240625_Dynr}
Are there \cms{s} $X_1$ and~$X_2$
and \mh{s} $T_1 \colon X_1 \to X_1$ and $T_2 \colon X_2 \to X_2$
such that the \cp{s} $C^* (\Z, X_1, T_1)$ and $C^* (\Z, X_2, T_2)$
have the same Elliott invariant and radius of comparison,
but are not isomorphic?
\end{pbm}

One might try to solve this problem by ``merging'' two dynamical
systems in a manner similar to what is done here with direct limits.
We don't know such a construction.

So far, we don't even know such examples
with the same Elliott invariant but different radii of comparison.

\bibliographystyle{plain}
\bibliography{nonisomorphic}

\end{document}